\documentclass[sn-mathphys,Numbered]{sn-jnl}
\usepackage{amsmath,amsthm,amsfonts,mathrsfs}
\usepackage{amssymb,braket,subfig}
\usepackage{tikz-cd}
\usepackage{xspace}
\usepackage{graphicx,booktabs}
\usepackage{paralist}
\usepackage{lmodern}
\usepackage{hyperref}
\hypersetup{%
    pdfmenubar=true,       
    pdffitwindow=true,     
    pdfstartview={FitW},    
    pdftitle={NKM},
    colorlinks=true,       
    linkcolor=red,          
    citecolor=red,        
    filecolor=magenta,      
    urlcolor=cyan,           
}
\usepackage{acro}
\usepackage{cleveref}
\usepackage{algorithm,algorithmic}
\usepackage{multirow}%
\usepackage[title]{appendix}%
\usepackage{xcolor}%
\usepackage{textcomp}%
\usepackage{manyfoot}%
\usepackage{listings}%
\makeatletter
\def\BState{\State\hskip-\ALG@thistlm}
\makeatother
\makeatletter
\def\BState{\State\hskip-\ALG@thistlm}
\makeatother
\Crefname{equation}{}{}
\usepackage{todonotes}
\usepackage{inputenc}
\usepackage{dsfont}

\newcommand{\bx}{\boldsymbol{x}}

\newcommand{\bz}{\boldsymbol{z}}
\newcommand{\bds}{\boldsymbol{s}}
\newcommand{\bda}{\boldsymbol{a}}
\newcommand{\bdf}{\boldsymbol{f}}
\newcommand{\bdg}{\boldsymbol{g}}
\newcommand{\bde}{\boldsymbol{e}}
\newcommand{\bdu}{\boldsymbol{u}}
\newcommand{\bdv}{\boldsymbol{v}}
\newcommand{\bdw}{\boldsymbol{w}}
\newcommand{\bzero}{\boldsymbol{0}}

\newcommand{\bdK}{\boldsymbol{K}}
\newcommand{\bF}{\boldsymbol{F}}
\newcommand{\bH}{\boldsymbol{H}}
\newcommand{\bzeta}{\boldsymbol{\zeta}}

\DeclareMathOperator*{\argmax}{arg\,max}
\DeclareMathOperator*{\diag}{\texttt{diag}}
\newcommand{\tra}{\boldsymbol{\mathsf{T}}}
\newcommand{\Real}{\mathbb{R}}
\theoremstyle{plain}
 \newtheorem{theorem}{Theorem}[section]
 \newtheorem{proposition}{Proposition}[section]
 \newtheorem{lemma}{Lemma}[section]
 \newtheorem{corollary}{Corollary}[section]
 
 \newtheorem{assumption}{Assumption}
 
\theoremstyle{definition}
 \newtheorem{definition}{Definition}

 \newtheorem{strategy}{Strategy}
 
\theoremstyle{remark}
  \newtheorem{remark}{Remark}[section]
\newcommand{\ie}{\textrm{i.e.}}

\begin{document}

\title[Convergence Analysis of Nonlinear Kaczmarz Method]{Convergence Analysis of Nonlinear Kaczmarz Method for Systems of Nonlinear Equations with Component-wise Convex Mapping}
\author[1,2]{\fnm{Yu} \sur{Gao}}
\author[1,2]{\fnm{Chong} \sur{Chen}}

\affil[1]{\orgdiv{LSEC, ICMSEC}, \orgname{Academy of Mathematics and Systems Science, Chinese Academy of Sciences}, \orgaddress{ \city{Beijing}, \postcode{100190}, \country{China}}}
\affil[2]{\orgname{University of Chinese Academy of Sciences}, \orgaddress{\city{Beijing}, \postcode{100190}, \country{China}}}


\abstract{
Motivated by a class of nonlinear imaging inverse problems, for instance, multispectral computed tomography (MSCT), this paper studies the convergence theory of the nonlinear Kaczmarz method (NKM) for solving the system of nonlinear equations with component-wise convex mapping, namely, the function corresponding to each equation being convex. However, such kind of nonlinear mapping may not satisfy the commonly used component-wise \textit{tangential cone condition} (TCC). For this purpose, we propose a novel condition named {\it relative gradient discrepancy condition} (RGDC), and make use of it to prove the convergence and even the convergence rate of the NKM with several general index selection strategies, where these strategies include cyclic strategy and maximum residual strategy. Particularly, we investigate the application of the NKM for solving nonlinear systems in MSCT image reconstruction. We prove that the nonlinear mapping in this context fulfills the proposed RGDC rather than the component-wise TCC, and provide a global convergence of the NKM based on the previously obtained results. Numerical experiments further illustrate the numerical convergence of the NKM for MSCT image reconstruction.
}


\maketitle

\section{Introduction} \label{sec:Introduction}
It is well-known that the Kaczmarz method and its variants have been widely used to solve large-scale systems of linear equations \cite{Kaczmarz1937,GORDON1970471,Tanabe1971,Strohmer2007ARK,Needell2010,liu16accelerated,Popa17_ConverRate,Greedy_kac_bai17,lu2022SGD}. It can gradually approximate the solution by alternating orthogonal projection onto the hyperplane composed of each linear equation. However, in practice, we often need to solve the  large-scale systems of nonlinear equations as below
\begin{align}\label{eq:nonlinear_system}
    \bF(\bx) = \bzero \Longleftrightarrow
    F_j(\bx) = 0, \quad \text{for}\ j = 1, \ldots, J. 
\end{align} 
Here $\bF(\bx) = [F_1(\bx),\ldots,F_J(\bx)]^{\tra} \in \Real^J$, $\bx \in \mathcal{X}\subseteq \Real^N$, $\bzero$ denotes the zero vector in $\Real^J$, and $F_j :\mathcal{X}\subseteq \Real^N \rightarrow \Real$ is a nonlinear real-valued function defined on the domain $\mathcal{X}$. 

Finding or approximating a solution to \cref{eq:nonlinear_system} is a significant issue in many research fields, including nonlinear optimization, numerical methods for nonlinear differential equations, and nonlinear inverse problems \cite{Iter_NonEq_book,Scherzer08}. An important extension of the Kaczmarz method, also known as the nonlinear Kaczmarz method (NKM), is first proposed in \cite{McCormick1975AnIP} for solving nonlinear system \cref{eq:nonlinear_system} as the following 
\begin{align}\label{eq:non_kacz}
    \bx^{k+1} = \bx^{k} - \frac{F_{j_k}(\bx^{k})}{\|\nabla F_{j_k}(\bx^{k})\|^2} \nabla F_{j_k}(\bx^{k}) \quad \text{for}~ k = 1, 2, \ldots, 
\end{align}
where index $j_k \in \{1, \ldots, J\}$ is selected by using a certain specified strategy. 

The equation corresponding to index $j_k$ in \cref{eq:nonlinear_system} can be regarded as the hypersurface 
\[
\left\{\bx \mid F_{j_k}(\bx) = 0\right\}.
\]
At point $\bx^{k}$, this hypersurface can be approximated by the hyperplane as  
\begin{align*}
    \left\{\bx \mid F_{j_k}(\bx^{k}) + \nabla F_{j_k}(\bx^{k})^{\tra}(\bx-\bx^{k}) = 0 \right\}.
\end{align*}
As a matter of fact, the iterative scheme of the NKM is generated by orthogonally projecting the current point $\bx^{k}$ onto the above hyperplane \cite{McCormick1975AnIP,non_kac_Mcco_77,NLKacz_22,Zhang2023OnSK,MR_NKM23,Lorenz23NBK}. 

In this work, we mainly study the convergence theory of the NKM for solving \cref{eq:nonlinear_system} with component-wise convex mapping. Note that the current article is concerned with the case when $J \ge N$ and \cref{eq:nonlinear_system} is solvable, and some solution is denoted as $\bx^*$.  

Several works have studied the numerical and/or theoretical aspects of the NKM. McCormick demonstrated the local convergence of the NKM within a sufficiently small neighborhood of the solution \cite{non_kac_Mcco_77}. The authors in \cite{NLKacz_22} proposed and analyzed a random strategy to the NKM, which selects the index by weighting its probability proportionally to the residual at the current iteration point. A linear convergence rate of the NKM with the maximum residual strategy is presented in \cite{MR_NKM23}. The authors in \cite{Zhang2023OnSK} employed and analyzed the sampling Kaczmarz--Motzkin strategy to the NKM for solving the nonlinear system, which is an extension of the linear case in \cite{SampKacMotz17}. In \cite{Lorenz23NBK}, the authors extended the idea of combining the Bregman projection with the Kaczmarz method for solving the linear systems in \cite{Lorenz14Bregman,Lorenz16RandSK} to the nonlinear case. The greedy random strategy proposed in \cite{Greedy_kac_bai17} for the linear systems was also applied to the nonlinear case in \cite{Zhang2022GreedyCN}. The authors in \cite{zhaozz14} proposed an extended algebraic reconstruction technique to solve the corresponding nonlinear system in image reconstruction of dual-energy CT, which is actually an application of the NKM. Furthermore, the NKM can be seen as a special case of the following Landweber--Kaczmarz algorithm 
\begin{align}\label{eq:Land_kacz}
    \bx^{k+1} = \bx^{k} - \omega_k F_{j_k}(\bx^{k}) \nabla F_{j_k}(\bx^{k}), 
\end{align}
where $\omega_k$ is referred to as the relaxation parameter at $\bx^{k}$.
The above algorithm with a cyclic strategy was first introduced and analyzed in \cite{Land_Kacz_Scherzer}. The authors in \cite{Martin06NewtonKac} combined \cref{eq:Land_kacz} with the iteratively regularized Gauss--Newton method, and then analyzed its convergence and regularization behavior using their proposed nonlinearity condition. A convergence analysis was presented in \cite{Averkac_Li18} for an averaged Kaczmarz iteration which can be seen as a hybrid method combining the Landweber iteration and the Kaczmarz method. A convergence analysis was given in \cite{sgd_for_ill_posed} for a stochastic gradient descent method which can be viewed as  \cref{eq:Land_kacz} with random strategy. The authors in \cite{wei15phasekac,Vershynin18phasekac,HM22KacPhase} applied the Kaczmarz method or its randomized version to solve the phase retrieval problem. 

The aforementioned studies are almost dedicated to understanding the theory or improving the performance of the NKM. However, most of the existing convergence results rely on the local \textit{tangential cone condition} (TCC), which is recalled in \cref{def:local_tcc}. The authors in \cite{hanke1995convergence} proposed the local TCC to analyze the convergence of the Landweber iteration for the nonlinear ill-posed problems. The authors in \cite{Land_Kacz_Scherzer} further analyzed the Landweber--Kaczmarz iteration in \cref{eq:Land_kacz} using a component-wise local TCC as described in \cref{def:local_tcc}. Additionally, the local TCC has also been employed in the analysis of \cite{sgd_for_ill_posed}. The component-wise local TCC was applied in \cite{Averkac_Li18, NLKacz_22,MR_NKM23,Zhang2023OnSK,Zhang2022GreedyCN,Lorenz23NBK}. 

Although the component-wise local TCC is commonly used to prove the convergence of the NKM, it may be hard to verify these conditions for some nonlinear mappings.
Moreover, the component-wise convex mapping, namely, each component function being convex (see \cref{def:convex_mapping}), may not satisfy the component-wise local TCC. Note that this type of examples is quite common in reality, such as the inverse problems (image reconstruction) in multispectral computed tomography (MSCT), discrete X-ray transform with nonlinear partial volume effect, phase retrieval, to just name a few \cite{gao2021EPD}. The authors in \cite{Zhang2023OnSK,Lorenz23NBK} also considered the case that the component functions of \cref{eq:nonlinear_system} are convex (or star-convex) but always nonnegative. However, as far as we know, there is almost no literature specifically analyzing the NKM for solving the general nonlinear system with component-wise convex mapping.
In this work, we focus on the convergence theory of the NKM for this case, and provide a new condition, named {\it relative gradient discrepancy condition} (RGDC), to perform the analysis. 
In particular, we will further apply our theoretical results to analyze the image reconstruction problem in MSCT.

This paper is organized as follows. The required preliminaries are presented in \cref{sec:Preliminaries}. In \cref{sec:conver_analy}, we begin with analyzing the limitations of component-wise local TCC for the case with component-wise convex mapping. For such case, we then introduce the RGDC and establish the convergence of the NKM. Section \ref{sec:NKM_conver_MSCT} analyzes the NKM for solving the specifically nonlinear systems in MSCT image reconstruction. We further show the numerical convergence of the NKM for that specific problem in \cref{sec:Numerical}. Finally, we conclude the paper in \cref{sec:Discussion}. 

\section{Mathematical preliminaries}
\label{sec:Preliminaries}

Here we give some required notation and preliminary results. 

\subsection{Notation}

Let $\Real^N$ be a finite-dimensional Euclidean space that equips with an inner product $\langle \bx_1, \bx_2 \rangle = \bx_1^{\tra}\bx_2$ and the norm $\|\cdot\| = \sqrt{\langle \cdot,\cdot\rangle}$. The $l_{\infty}$ norm is defined as $\|\bx\|_{\infty} := \max\limits_{i\in \{1, \ldots, N\}} |x_i|$.
For a matrix, $A= (a_{ji})\in\Real^{M\times N}$, the matrix norms are defined as follows
\begin{align*}
    \|A\| := \max\limits_{\|\bx\|=1} \|A \bx\|,\quad \|A\|_{\infty} := \max\limits_{\|\bx\|_{\infty}=1} \|A \bx\|_{\infty},\quad \|A\|_F := \sqrt{ \sum_{j=1}^M\sum_{i=1}^{N} a_{ji}^2}.
\end{align*}
Let $\sigma_{\min}(A)$, $Null(A)$, $Range(A)$ and $A^{\dagger}$ be the minimum nonzero singular value, null (kernel) space, range, and Moore--Penrose pseudoinverse of $A$, respectively. If $A$ has full column rank, define the scaled condition number as 
\[
\kappa_F(A):=\|A\|_F \|A^{\dagger}\|=\frac{\|A\|_F}{\sigma_{\min}(A)}.
\] 
We denote $\diag(\bx)$ the diagonal matrix with diagonal elements as $\bx$, $\#\alpha$ the number of elements in set $\alpha$.
Let $A[\alpha,\beta]$ denote the submatrix of $A$ lying in rows $\alpha$ and columns $\beta$. 
For the square matrix $A$, denote $\lambda(A),\det(A)$ and $\textrm{tr}(A)$ as the set of eigenvalues, determinant and trace of $A$, respectively. 

If the $\bF(\bx)$ defined in \cref{eq:nonlinear_system} is continuously differentiable, we refer to its Jacobian matrix as $\bF^{\prime}(\bx) := [\nabla F_1(\bx),\ldots, \nabla F_J(\bx)]^{\tra}\in\Real^{J\times N}$.
We indicate $\mathcal{B}_{\rho}(\bx)$ the open ball with radius $\rho$ around a given point $\bx$. 
Define the level set of $F_j$ at value $F_j(\bar{\bx})$ as 
\begin{equation*}
    \textbf{Lev}_j(\bar{\bx}) := \left\{ \bx \mid F_j(\bx) = F_j(\bar{\bx})\right\}.
\end{equation*}
Note that when $F_j$ is convex, its sublevel set $\left\{ \bx \mid F_j(\bx) \le t,  t\in\Real\right\}$ is convex. Moreover, the tangent plane of $F_j$ at $\bar{\bx}$ is given by 
\begin{equation*}
 \textbf{T}_j(\bar{\bx}):=\left\{\bx \mid \nabla F_j(\bar{\bx})^{\tra} (\bx-\bar{\bx}) = 0 \right\}.
\end{equation*}

\subsection{Preliminary results}

The local \textit{tangential cone condition} (TCC) in $\Real^N$ is given as follows. 
\begin{definition}(\cite{Scherzer08})\label{def:local_tcc}
Let $\bF$ be a continuously differentiable mapping. The $\bF$ satisfies local TCC in $\Omega\subseteq\Real^N$
    if there exists $\eta < 1/2$ such that
\begin{equation}\label{eq:local_tcc_mapping}
    \left\|\bF(\bx_1)-  \bF(\bx_2)- \bF^{\prime}(\bx_1)(\bx_1-\bx_2)\right\| 
    \leq \eta \left\|\bF(\bx_1)-\bF(\bx_2)\right\|, 
\end{equation}
$\forall \bx_{1}, \bx_{2} \in \Omega$. 
Moreover, the $\bF$ satisfies the component-wise local TCC in $\Omega\subseteq\Real^N$ if for any component $F_j$, there exists $\eta_j < 1/2$ such that 
\begin{equation}\label{eq:local_tcc_component}
\left|F_j(\bx_1)-  F_j\left(\bx_2\right)- \nabla F_j (\bx_1)^{\tra}(\bx_1-\bx_2)\right| \leq \eta_j \left|F_j(\bx_1)-F_j(\bx_2)\right|, 
\end{equation}
$\forall \bx_{1}, \bx_{2} \in \Omega$. 
\end{definition}

It is easy to show that the component-wise local TCC is stronger than the local TCC. In addition, the authors in \cite{Scherzer08} presented the result below for the local TCC within ball $\mathcal{B}_{\rho}(\bx_0)$.
\begin{proposition}(\cite{Scherzer08})\label{prop:local_tcc_property}
    Suppose that $\bF$ satisfies local TCC \cref{eq:local_tcc_mapping} in $\mathcal{B}_{\rho}(\bx_0)$.
    Then for any $\bar{\bx}\in\mathcal{B}_{\rho}(\bx_0)$, 
    \begin{equation*}
        \{\bx\mid \bF(\bx)=\bF(\bar{\bx}) \} \cap \mathcal{B}_{\rho}(\bx_0)  = \bar{\bx} + Null(\bF^{\prime}(\bar{\bx}))\cap\mathcal{B}_{\rho}(\bx_0),
    \end{equation*}
    and $Null(\bF^{\prime}(\bx))= Null(\bF^{\prime}(\bar{\bx}))$ for $\bx\in \{\bx\mid \bF(\bx)=\bF(\bar{\bx}) \} \cap \mathcal{B}_{\rho}(\bx_0) $.
\end{proposition}

Moreover, we introduce the following definition.

\begin{definition}(\cite{Iter_NonEq_book})\label{def:convex_mapping}
    Mapping $\bF(\bx)$ is component-wise convex in a convex set $\Omega$ if 
    \begin{equation*}
        \bF(\lambda \bx_1 + (1-\lambda)\bx_2)\le \lambda \bF(\bx_1) + (1-\lambda) \bF(\bx_2), 
    \end{equation*}
$\forall \bx_{1}, \bx_{2} \in \Omega$ and $\forall\lambda\in (0,1)$,
    where $\le$ denotes the component-wise partial ordering. The $\bF(\bx)$ is strictly component-wise convex in $\Omega$ if the above inequality strictly holds whenever $\bx_1\neq\bx_2$.
\end{definition} 

\section{Convergence analysis}\label{sec:conver_analy}

In this section, we shall analyze the convergence of the NKM under some new proposed condition rather than the often used TCC. To begin with, we present the general NKM in \cref{alg:nonlinear_kacz}.
\begin{algorithm}[htbp]
    \caption{The general NKM.}
    \label{alg:nonlinear_kacz}
    \begin{algorithmic}[1]
   \STATE \emph{Initialization}: Given mapping $\bF$, and the initial point $\bx^0$. Let $k \gets 0$.  
    \STATE \emph{Loop}: 
    \STATE \quad  Select an index $j_k\in\{1,\ldots,J\}$ according to some strategy.
    \STATE  \quad Update $\bx^{k+1}$ by \cref{eq:non_kacz}. 
    \STATE \quad \textbf{If} some given termination condition is satisfied, \textbf{output} $\bx^{k+1}$; 
    \STATE \quad \textbf{otherwise}, let $k \gets k+1$, \textbf{goto} \emph{Loop}.
    \end{algorithmic}
\end{algorithm}

\subsection{A brief analysis for the component-wise local TCC}
\label{subsec:analy_tcc}
In this subsection, we analyze the component-wise local TCC in $\Real^N$, and have the following result. 
\begin{proposition}\label{thm:level_set_tangent_plane}
    Suppose that $\bF$ satisfies component-wise local TCC \cref{eq:local_tcc_component} in $\mathcal{B}_{\rho}(\bx_0)$.
    Then for all $\bar{\bx}\in\mathcal{B}_{\rho}(\bx_0)$, the component function $F_j$ satisfies
    \begin{equation*}
        \normalfont \textbf{Lev}_j(\bar{\bx})\cap\mathcal{B}_{\rho}(\bx_0) = \textbf{T}_j(\bar{\bx})\cap\mathcal{B}_{\rho}(\bx_0),
    \end{equation*}
    and $Null(\nabla F_j(\bx))=Null(\nabla F_j(\bar{\bx}))$ for $\normalfont\bx\in\textbf{Lev}_j(F_j(\bar{\bx}))\cap\mathcal{B}_{\rho}(\bx_0)$.
\end{proposition}
\begin{proof}
    The result can be directly obtained by \cref{prop:local_tcc_property}.
\end{proof}

If there exists $\bx^*\in\mathcal{B}_{\rho}(\bx_0)$ such that $\bF(\bx^*) = \bzero$, then using \cref{thm:level_set_tangent_plane}, we have
\begin{equation*}
    \textbf{Lev}_j(\bx^*)\cap\mathcal{B}_{\rho}(\bx_0) = \textbf{T}_j(\bx^*)\cap\mathcal{B}_{\rho}(\bx_0).
\end{equation*}
In other words, it shows that the nonlinear system in \cref{eq:nonlinear_system} degenerates to the following linear system
\begin{equation*}
    \bF^{\prime}(\bx^*)(\bx-\bx^*) = \bzero
\end{equation*}
in $\mathcal{B}_{\rho}(\bx_0)$. Thus implementing the nonlinear Kaczmarz iteration \cref{eq:non_kacz} from any initial point in $\mathcal{B}_{\rho}(\bx_0)$ is equivalent to perform the classical Kaczmarz method to solve the linear system above.
Hence, the convergence results of the Kaczmarz method are still valid to the NKM for this case.  

\begin{corollary}\label{thm:nonzero_curvature_tcc}
   Assume that there exists an $F_j$ such that the mean curvature $\nabla \cdot\left(\nabla F_j / \|\nabla F_j\|\right)$ of the level set is nonzero at some point  in $\mathcal{B}_{\rho}(\bx_0)$. Then $\bF$ does not satisfy the component-wise local TCC in $\mathcal{B}_{\rho}(\bx_0)$.
\end{corollary}
\begin{proof}
We prove the result by a contradiction. Assume that
$\bF$ satisfies the component-wise local TCC in $\mathcal{B}_{\rho}(\bx_0)$.
By \cref{thm:level_set_tangent_plane}, the unit normal vectors $\nabla F_j/\|\nabla F_j\|$ of the hypersurface $\textbf{Lev}_j$ at the positions in $\mathcal{B}_{\rho}(\bx_0)$ are identical. Hence, the corresponding mean curvature of the hypersurface vanishes, which contradicts with the above assumption.
\end{proof}

As we can see from \cref{thm:nonzero_curvature_tcc}, there are some nonlinear mappings that do not satisfy the component-wise local TCC. If there exists a strictly convex component of mapping $\bF$ in $\mathcal{B}_{\rho}(\bx_0)$, for instance, $F_j$, then 
\begin{multline}\label{eq:curvature}
        \nabla \cdot \left( \frac{\nabla F_j(\bx)}{\|\nabla F_j(\bx)\|} \right) \\
        =\frac{1}{\|\nabla F_j(\bx)\|} \left(\textrm{tr}(\nabla^2 F_j(\bx)) - \frac{\nabla F_j(\bx)^{\tra}}{\|\nabla F_j(\bx)\|} \nabla^2 F_j(\bx)  \frac{\nabla F_j(\bx)}{\|\nabla F_j(\bx)\|} \right) > 0. 
\end{multline}
By \cref{thm:nonzero_curvature_tcc}, the $\bF$ does not satisfy the component-wise local TCC. 
Moreover, we will demonstrate that the mappings in MSCT fail to satisfy the component-wise local TCC under the given conditions in \cref{thm:map_MSCT_not_tcc}.

\subsection{The proposed condition}\label{subsec:RGDC_lemmas}

In this subsection, we put forward the {\it relative gradient discrepancy condition} (RGDC) to characterize the nonlinearity of a mapping. 

\begin{assumption}\label{ass:relative_grad_discrepancy}
    For every $j=1,\ldots,J$, the gradient of $F_j$ satisfies
    \begin{align*}
    \| \nabla F_j(\bx_1) - \nabla F_j(\bx_2) \|  \le \gamma \|\nabla F_j(\bx_1)\|, \quad \forall\ \bx_1, \bx_2\in\Omega\subseteq \Real^N 
    \end{align*}
with $0\le\gamma<1$.
\end{assumption}
\begin{remark}\label{remark:ass_grad}
When $F_j$ is linear, the \cref{ass:relative_grad_discrepancy} holds trivially. If there exists $\bar{\bx}$ such that $\|\nabla F_j(\bar{\bx})\| = 0$ for some $j$, then the above assumption implies that $\|\nabla F_j(\bx)\| = 0$ for every $\bx\in\Omega$. 
\end{remark}

\begin{remark}\label{remark:ass_grad2}
If $\nabla F_j$ is Lipschitz continuous in $\mathcal{B}_{\rho}(\bx_0)$, and $\|\nabla F_j(\bx)\|$ has a positive lower bound in $\mathcal{B}_{\rho}(\bx_0)$, then
\begin{align*}
\| \nabla F_j(\bx_1) - \nabla F_j(\bx_2) \| \le \textbf{Lip}(\nabla F_j) \| \bx_1 - \bx_2 \| \le
\frac{\textbf{Lip}(\nabla F_j)\rho}{\inf\limits_{\bx} \|\nabla F_j(\bx)\|} \| \nabla F_j(\bx_1)\|,
\end{align*}
where $\textbf{Lip}(\nabla F_j)$ is the Lipschitz constant of $\nabla F_j$ in $\mathcal{B}_{\rho}(\bx_0)$. This shows that the RGDC can be established in a sufficiently small region. 
\end{remark}

Moreover, it is possible that a certain mapping satisfies the RGDC in a relatively large set or even in a global sense. In \cref{lem:general_verify_grad,lem:specific_verify_grad}, we will prove that the mapping in MSCT reconstruction fulfills the RGDC globally under proper conditions.
Intuitively, not only does the RGDC imply that the gradient directions at different points do not differ too much, but also the magnitudes of gradients are close. 
More precisely, we have the following result.
\begin{proposition}\label{lemma:RGDC_bound}
Let the \ref{ass:relative_grad_discrepancy} hold in a set $\Omega$ for $\gamma \in [0, 1)$. For every component function $F_j$, if $\bx_1,\bx_2\in \Omega$, then
\begin{equation}\label{eq:angle_bounded}
    \frac{\langle \nabla F_j(\bx_1), \nabla F_j(\bx_2) \rangle}{\|\nabla F_j(\bx_1)\|\cdot \|\nabla F_j(\bx_2)\|} \ge  1-\frac{\gamma^2}{2},
\end{equation}
    and 
\begin{equation}\label{eq:length_bounded}
   (1-\gamma) \|\nabla F_j(\bx_2)\| \le \| \nabla F_j(\bx_1) \|\le (1+\gamma) \|\nabla F_j(\bx_2)\|.
\end{equation}
\end{proposition}
\begin{proof}
    See \cref{proof:RGDC_bounded} for the proof.
\end{proof}

We draw an example diagram of the RGDC in \cref{fig:RGDC} to show the results of \cref{lemma:RGDC_bound} in the two-dimensional case. 
\begin{figure}[ht]
    \centering
    \includegraphics[width=0.6\textwidth]{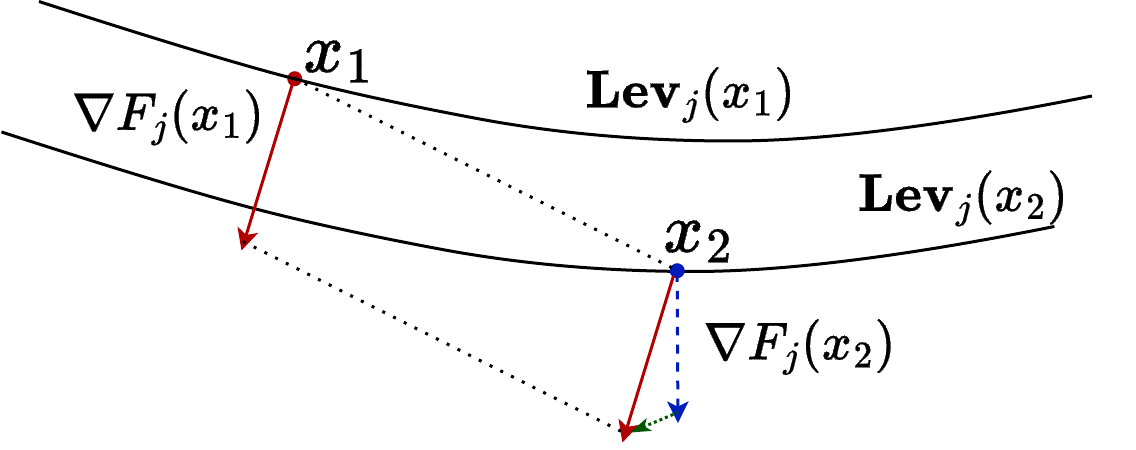}
    \vspace{3mm}
    \caption{The diagram depicts an example which fulfills the \ref{ass:relative_grad_discrepancy}.}
    \label{fig:RGDC}
\end{figure}
To further understand the RGDC, we have the following result.
\begin{proposition}\label{lemma:RGDC_property}
Let the \ref{ass:relative_grad_discrepancy} hold in a convex set $\Omega$ for $\gamma \in [0 ,1)$. For any $\bx_1,\bx_2\in \Omega$, if $\bF^{\prime}(\bx_1)$ has full column rank, then
\begin{equation}\label{eq:derived_nonlin_condi}
\left\|\bF(\bx_1) - \bF(\bx_2)-\bF^{\prime}(\bx_1) (\bx_1-\bx_2)\right\| \le \gamma\kappa_F(\bF^{\prime}(\bx_1))\left\|\bF^{\prime}(\bx_1 ) (\bx_1-\bx_2)\right\|.
\end{equation}
\end{proposition}
\begin{proof}
    See \cref{proof:RGDC} for the proof.
\end{proof}

If $\bF$ is linear, the condition of full column rank guarantees the uniqueness of the solution to \cref{eq:nonlinear_system}. Based on \cref{lemma:RGDC_property}, we can generalize such uniqueness to the nonlinear case of \cref{eq:nonlinear_system} when the mapping satisfies the RGDC as the following result.

\begin{proposition}\label{lemma:unique_solution}
    Let the \ref{ass:relative_grad_discrepancy} hold in a convex set $\Omega$ for $\gamma \in [0 ,1)$. Suppose that the nonlinear system \cref{eq:nonlinear_system} has a solution $\bx^*\in\Omega$ and $\bF^{\prime}(\bx^*)$ has full column rank. If $\gamma\kappa_F(\bF^{\prime}(\bx^*)) < 1$, then the solution is unique in $\Omega$.
\end{proposition}
\begin{proof}
    See \cref{proof:RGDC_unique} for the proof.
\end{proof}

Next we discuss the relationship between 
\cref{eq:derived_nonlin_condi} in 
\cref{lemma:RGDC_property} and the local TCC \cref{eq:local_tcc_mapping}.
On the one hand, under the conditions of 
\cref{lemma:RGDC_property}, 
\cref{eq:derived_nonlin_condi} holds for a special point $\bx_1$ and any point $\bx_2$ in convex set $\Omega$. 
Using triangle inequality, we obtain 
\begin{multline*}
    \gamma\kappa_F(\bF^{\prime}(\bx_1)) \left\|\bF^{\prime}(\bx_1 ) (\bx_1-\bx_2)\right\| \ge \left\|\bF(\bx_1) - \bF(\bx_2)-\bF^{\prime}(\bx_1 ) (\bx_1-\bx_2)\right\|\\
    \ge \left\|\bF^{\prime}(\bx_1 ) (\bx_1-\bx_2)\| - \|\bF(\bx_1) - \bF(\bx_2)\right\|.
\end{multline*}
Arranging the above terms and using 
\cref{eq:derived_nonlin_condi} imply that
\begin{align*}
    \left\|\bF(\bx_1) - \bF(\bx_2)-\bF^{\prime}(\bx_1 ) (\bx_1-\bx_2)\right\| \le \frac{\gamma\kappa_F(\bF^{\prime}(\bx_1))}{1-\gamma\kappa_F(\bF^{\prime}(\bx_1))} \left\|\bF(\bx_1) - \bF(\bx_2)\right\|.
\end{align*}
If the $\gamma$ further satisfies $\gamma\kappa_F(\bF^{\prime}(\bx_1))<1/3$, the above inequality becomes the formula in \cref{eq:local_tcc_mapping} for  
\begin{equation*}
    \eta = \frac{\gamma\kappa_F(\bF^{\prime}(\bx_1))}{1-\gamma\kappa_F(\bF^{\prime}(\bx_1))} < \frac{1}{2}.
\end{equation*}
Note that this requires the additional condition that  $\bF^{\prime}(\bx_1)$ has full column rank. However, the local TCC is established
for any point within $\Omega$.

On the other hand, when \cref{eq:local_tcc_mapping} holds, we have
\begin{equation*}
    \left\|\bF(\bx_1) - \bF(\bx_2)-\bF^{\prime}(\bx_1 ) (\bx_1-\bx_2)\right\| \le \frac{\eta}{1-\eta} \left\|\bF^{\prime}(\bx_1 ) (\bx_1-\bx_2)\right\|.
\end{equation*}
It is similar to the formula in \cref{eq:derived_nonlin_condi}.

\subsection{Convergence analysis of the NKM}

In this subsection, we provide the convergence analysis of the NKM for solving \cref{eq:nonlinear_system} when $\bF$ is component-wise convex and satisfies the \ref{ass:relative_grad_discrepancy}. 
Firstly, a result of the general NKM (\cref{alg:nonlinear_kacz}) is presented as below.
\begin{proposition}\label{lem:NKM_convex}
    If $\bF(\bx)$ is component-wise convex in $\mathbb{R}^N$, for any given initial point, the iteration point $\bx^{k+1}$ generated by \cref{alg:nonlinear_kacz} fulfills $F_{j_k}(\bx^{k+1})\ge0$. 
\end{proposition} 
\begin{proof}
    See \cref{proof:NKM_convex} for the proof.
\end{proof}

The result in \cref{lem:NKM_convex} demonstrates that when applying the NKM to solve the nonlinear system with component-wise convex mapping, it is guaranteed that none of the iteration points will fall within the set $\{\bx \mid F_j(\bx) < 0, j=1,\ldots,J \}$, regardless of what the initial point or selection strategy is employed. 

To proceed, we propose a general index selection strategy as follows. 
\begin{strategy}\label{stra:non_negative}
Select an index $j_k$ such that $F_{j_{k}}(\bx^k) >0$.
\end{strategy}
According to \cref{lem:NKM_convex}, we have $F_{j_k}(\bx^{k+1}) \ge0$, indicating that there almost exists feasible index for strategy \ref{stra:non_negative}. Furthermore, we have the following results for the NKM using \cref{stra:non_negative}.
\begin{lemma}\label{lem:dist_decrease}
    Let $\{\bx^{k}\}$ be the sequence generated by \cref{alg:nonlinear_kacz} using \cref{stra:non_negative}. 
    Suppose that $\bF(\bx)$ is component-wise convex in $\mathbb{R}^N$, then 
    \begin{align*}
        \|\bx^{k+1}-\bx^*\| \le \|\bx^{k}-\bx^*\|, \quad \forall k\ge0.
    \end{align*}
\end{lemma} 
\begin{proof}
    See \cref{proof:dist_decrease} for the proof.
\end{proof}

For the case of \cref{lem:dist_decrease},
the iterative scheme in \cref{eq:non_kacz} is equivalent to solve a convex feasibility problem, which was also studied in \cite{Lorenz23NBK,MR_NKM23}. 
\begin{lemma}\label{lem:RGDC_result}
Suppose that the conditions in \cref{lem:dist_decrease} hold and \ref{ass:relative_grad_discrepancy} holds in $\mathcal{B}_{\rho}(\bx^*)$ for $\gamma  \in [0, 1)$. If $\bx^0\in\mathcal{B}_{\rho}(\bx^*)$, then
\begin{equation*}
    |F_{j_k}(\bx^{k+1}) | \le c(\gamma) |F_{j_k}(\bx^{k}) |, \quad \forall k\ge0, 
\end{equation*}
where $c(\gamma)= (\gamma+\gamma^2/2-\gamma^3/2) \in [0, 1)$. 
\end{lemma}
\begin{proof}
    See \cref{proof:RGDC_result} for the proof.
\end{proof}

Even though both of the distance and residual decrease as shown in \cref{lem:dist_decrease} and \cref{lem:RGDC_result}, we still can not conclude that the sequence generated by the NKM converges to the solution. For example, consider the scenario that there exists $k_0$ such that $j_k = j_{k_0}$ for $k \ge k_0$. As in the previous analysis, such selected index always fulfills strategy \ref{stra:non_negative}. Then, by \cref{lem:RGDC_result},
\begin{equation*}
    |F_{j_{k_0}}(\bx^k)|\le c(\gamma) |F_{j_{k_0}}(\bx^{k-1})|\le \cdots \le c(\gamma)^{k-k_0} |F_{j_{k_0}}(\bx^{k_0})|.
\end{equation*}
Therefore, $F_{j_{k_0}}(\bx^k) \rightarrow 0$ as $k\rightarrow \infty$. But the iteration sequence may not converge to the solution of the entire system of equations, as shown in \cref{fig:stra1}.
\begin{figure}[htbp]
    \centering
    \includegraphics[width=0.49\textwidth]{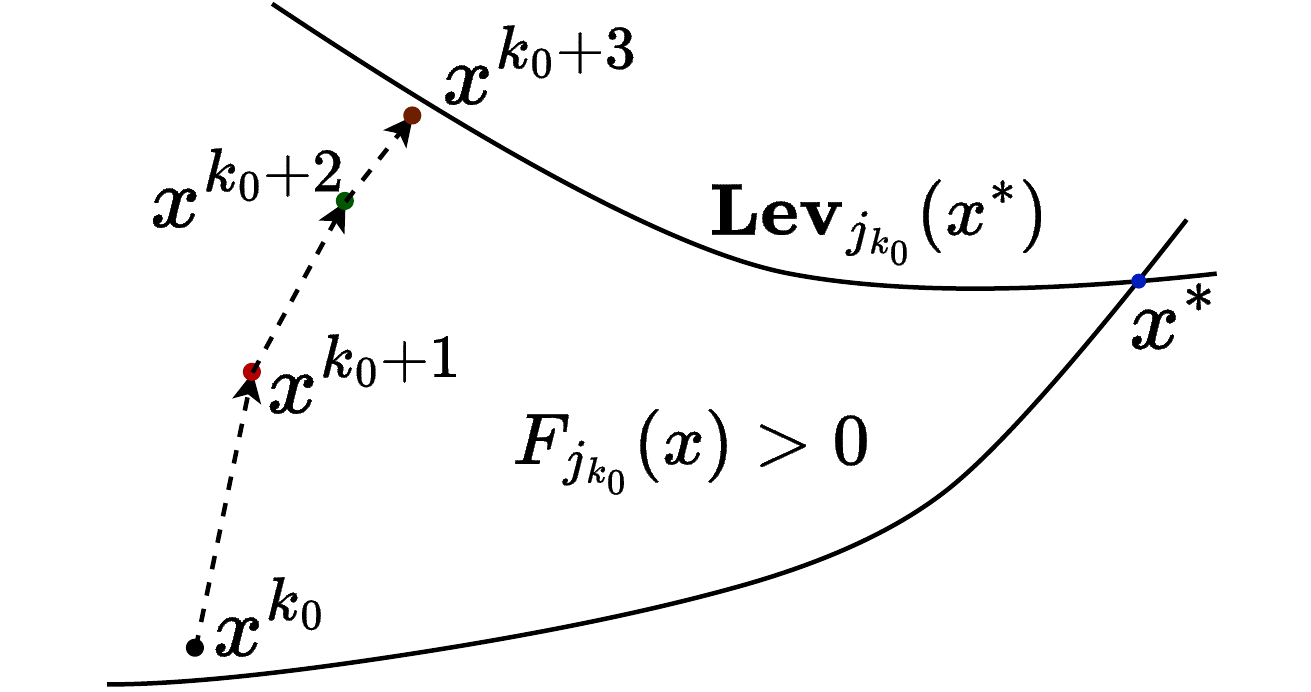}
        \vspace{3mm}
    \caption{The figure depicts an example for the iterated procedure of the NKM using \cref{stra:non_negative}.}
    \label{fig:stra1}
\end{figure}

To consider all the equations in the system, we introduce another strategy as below. 
\begin{strategy}\label{stra:infinite_often}
    For every $j=1,\ldots,J$, and for any positive integer $n$, there always exists $k\ge n$ such that $j_k=j$. Namely, each component in \cref{eq:nonlinear_system} is selected infinitely.
\end{strategy}
Obviously, the cyclic strategy is a specific example satisfying \cref{stra:infinite_often}, whose index is selected as
\begin{equation}\label{eq:kacz_cyclic}
    j_k^{\text{cyclic}} := (k\ \textbf{mod}\ J) + 1.
\end{equation}

Combining with \cref{stra:infinite_often}, we can prove the following convergence result. 
\begin{theorem}\label{thm:nkm_stra12_converge}
Let $\{\bx^k\}$ be the sequence generated by \cref{alg:nonlinear_kacz} using the strategy satisfying both strategies  \ref{stra:non_negative} and \ref{stra:infinite_often}.
Suppose that $\bF(\bx)$ is component-wise convex. Then the sequence converges to a solution of \cref{eq:nonlinear_system}.\\
\indent Furthermore, assume that $\bF^{\prime}(\bx^*)$ has full column rank and the \ref{ass:relative_grad_discrepancy} holds in $\mathcal{B}_{\rho}(\bx^*)$ for $\gamma \in [0, 1)$ such that $\gamma\kappa_F(\bF^{\prime}(\bx^*))<1$.
If the initial point $\bx^0\in\mathcal{B}_{\rho}(\bx^*)$, then the sequence converges to
the unique solution $\bx^*$ in $\mathcal{B}_{\rho}(\bx^*)$.
\end{theorem}

\begin{proof}
According to the result in \cref{lem:dist_decrease}, the decrease of the distance from the iteration point to the solution $\bx^*$ is at least
\begin{equation}\label{eq:decrease_dist_ge0}
 \| \bx^{k+1} -\bx^* \|^2- \| \bx^{k}-\bx^* \|^2 \le -\frac{|F_{j_k}(\bx^{k})|^2  }{\|\nabla F_{j_k}(\bx^{k})\|^2}.
\end{equation}
Hence, summing it over $k=1,2,\ldots,\widetilde{k}$ yields
\begin{equation*}
    \|\bx^0 - \bx^*\|^2 - \|\bx^{\widetilde{k}} - \bx^*\|^2\ge \sum_{k=0}^{\widetilde{k}-1} \frac{|F_{j_k}(\bx^{k})|^2 }{\|\nabla F_{j_k}(\bx^{k})\|^2}.
\end{equation*}
Let $\widetilde{k}\rightarrow\infty$, the series on the right-hand side of the above inequality converges. Then $\forall \epsilon>0$, there exists $k$ such that $\forall k_2>k_1\ge k(\epsilon)$
\begin{align*}
    \|\bx^{k_2} - \bx^{k_1}\|^2 = \left\|\sum_{k=k_1}^{k_2 -1} \frac{F_{j_k}(\bx^{k})}{\|\nabla F_{j_k}(\bx^{k})\|^2} \nabla F_{j_k}(\bx^{k}) \right\|^2 \le \sum_{k=k_1}^{k_2 -1}
    \frac{(F_{j_k}(\bx^{k}))^2}{\|\nabla F_{j_k}(\bx^{k})\|^2} \le \epsilon.
\end{align*} 
This illustrates that $\{\bx^k\}$ is a Cauchy sequence. Hence, it converges to some point $\bar{\bx}^*$. Moreover, we can conclude that $\|\nabla F_{j_k}(\bx^k)\|$ is upper bounded for all $k\ge0$. Then, by the last inequality of the above formula, $\forall \varepsilon>0$, there exists $n$ such that
\begin{equation*}
    |F_{j_k}(\bx^k)|  \le \varepsilon \quad \text{as} ~ k\ge n.
\end{equation*}
By \cref{stra:infinite_often}, for every $l=1, \ldots, J$, there exists $k_l\ge n$ such that $j_{k_l}=l$. Therefore, we have
\begin{equation*}
    |F_{l}(\bx^{k_l})| = |F_{j_{k_l}}(\bx^{k_l})|  \le \varepsilon.
\end{equation*} 
Together with the continuity of $F_l$, we obtain $|F_l(\bar{\bx}^*)|=0$. Hence, $\bar{\bx}^*$ is also a solution of the nonlinear system. 

Furthermore, by \cref{lemma:unique_solution}, we have $\bar{\bx}^* = \bx^*$. The proof is completed.
\end{proof}

As we know, the strategies  \ref{stra:non_negative} and \ref{stra:infinite_often} can be simultaneously satisfied. To illustrate this, we give an intuitive example in \cref{fig:stra1_2}.  
\begin{figure}[htbp]
    \centering
    \includegraphics[width=0.49\textwidth]{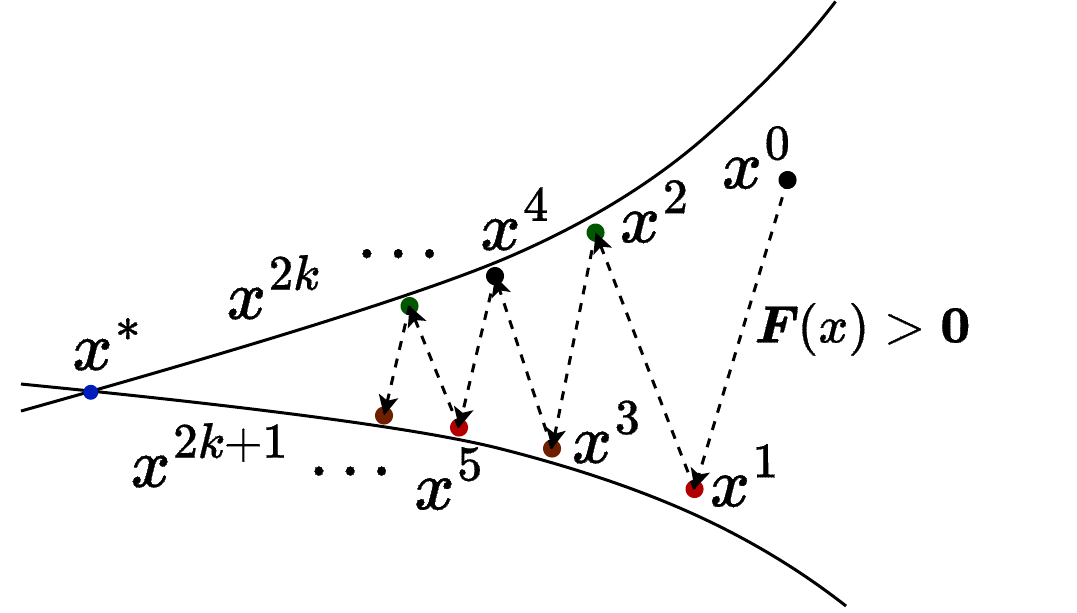}
        \vspace{3mm}
    \caption{The figure depicts an example for the iterated procedure of the NKM satisfying strategies \ref{stra:non_negative} and \ref{stra:infinite_often}. } 
    \label{fig:stra1_2}
\end{figure}

Furthermore, we study a generalized maximum residual strategy as follows.
\begin{strategy}\label{stra:theta_res}
For given $\theta \in (0, 1/\sqrt{J}]$, select an index $j_k$ such that
\begin{equation}\label{eq:theta_res}
    |F_{j_{k}}(\bx^k)| \ge \theta \left\|\bF(\bx^k) \right\|. 
\end{equation}
\end{strategy}
\begin{remark}\label{rem:maximum_residual}
Note that the maximum residual strategy \cite{non_kac_Mcco_77} fulfills \cref{stra:theta_res} for $\theta = 1/\sqrt{J}$, whose index is selected as
\begin{align}\label{eq:kacz_max_res}
   j_k^{\text{maxres}} := \argmax\limits_{1\le j\le J} |F_j(\bx^{k}) - F_j(\bx^*)|.
\end{align}
\end{remark}

Importantly, we can prove the linear convergence of the NKM using \cref{stra:theta_res}.
\begin{theorem}\label{thm:nkm_converge_rate} 
    Let $\{\bx^{k}\}$ be the sequence generated by \cref{alg:nonlinear_kacz} using strategy \ref{stra:theta_res} for $\theta \in (0, 1/\sqrt{J}]$. Suppose that in $\mathcal{B}_{\rho}(\bx^*)$, $\bF(\bx)$ is component-wise convex, $\bF^{\prime}(\bx^*)$ has full column rank and the \ref{ass:relative_grad_discrepancy} holds for $\gamma \in [0, 1)$ such that
    \begin{equation}\label{eq:condi_gamma_le}
        \gamma\kappa_F(\bF^{\prime}(\bx^*)) \le \frac{\theta(1-\tau)}{ 2+\theta(1-\tau)}
    \end{equation}
    for $0<\tau<1$.
    If the initial point $\bx^0\in\mathcal{B}_{\rho}(\bx^*)$, then the sequence converges to the unique solution $\bx^*$ at least with a linear rate. 
\end{theorem}

\begin{proof}
We first prove the following monotonicity result when $\bx^k\in\mathcal{B}_{\rho}(\bx^*)$,
\begin{equation*}
    \| \bx^{k+1} -\bx^* \|\le \| \bx^{k}-\bx^* \|.
\end{equation*}
Depending on the sign of $F_{j_k}(\bx^{k})$, there are two cases.
\paragraph{Case 1: $F_{j_k}(\bx^{k})\ge 0$.}   
By \cref{lem:dist_decrease}, the monotonicity result holds.
\paragraph{Case 2: $F_{j_k}(\bx^{k})< 0$.} 
By the derivation in \cref{lem:dist_decrease}, we have
\begin{multline*}
\| \bx^{k+1} -\bx^* \|^2- \| \bx^{k}-\bx^* \|^2 \\
= \frac{ F_{j_k} (\bx^{k} )}{ \|\nabla F_{j_k} (\bx^{k} ) \|^2} \Big\{ F_{j_k} (\bx^{k}) -F_{j_k} (\bx^*)-2\nabla F_{j_k} (\bx^{k} )^{\tra} (\bx^{k}-\bx^* )\Big\}.
\end{multline*}
In what follows, we prove the following two statements 
\begin{itemize}
    \item[(i)] $\nabla F_{j_k} (\bx^{k} )^{\tra} (\bx^{k}-\bx^* )<0$;
    \item[(ii)] $(1 + \tau)|F_{j_k} (\bx^{k}) -F_{j_k} (\bx^*)| \le 2|\nabla F_{j_k} (\bx^{k} )^{\tra} (\bx^{k}-\bx^* )|$.
\end{itemize}
Based on the two statements, combining with the negativity of $F_{j_k}(\bx^{k})$,
the monotonicity result can be established. 

To confirm the above statements, we split $\nabla F_{j_k} (\bx^{k} )^{\tra} (\bx^{k}-\bx^* )$ into two parts as 
\begin{multline*}
    \nabla F_{j_k} (\bx^{k} )^{\tra} (\bx^{k}-\bx^* )\\ =\underbrace{\left[\nabla F_{j_k} (\bx^{k} ) - \nabla F_{j_k} (\bx^* )\right]^{\tra} (\bx^{k}-\bx^* )}_{\textrm{Part}\ 1}  + \underbrace{\nabla F_{j_k} (\bx^* )^{\tra} (\bx^{k}-\bx^* )}_{\textrm{Part}\ 2}.
\end{multline*}
Using the convexity, Cauchy--Schwarz inequality and the \ref{ass:relative_grad_discrepancy} imply that
\begin{equation}\label{eq:estimate_part1}
    0\le \textrm{Part}\ 1 \le \|\nabla F_{j_k} (\bx^{k} ) - \nabla F_{j_k} (\bx^* )\| \|\bx^{k}-\bx^*\| \le \gamma 
    \|\nabla F_{j_k} (\bx^* )\| \|\bx^{k}-\bx^*\|.
\end{equation}
On the other hand, again by the convexity, we have
\begin{align*}
    F_{j_k} (\bx^{k}) -F_{j_k} (\bx^*)-\nabla F_{j_k} (\bx^* )^{\tra} (\bx^{k}-\bx^* )\ge 0.
\end{align*}
Recall that $F_{j_k} (\bx^{k}) -F_{j_k} (\bx^*)<0$, thus the following results hold 
\begin{equation}\label{eq:sign_part2}
    \textrm{Part}\ 2 = \nabla F_{j_k} (\bx^* )^{\tra} (\bx^{k}-\bx^* )<0,
\end{equation}
and 
\begin{equation}\label{eq:estimate_part2}
|\textrm{Part}\ 2| = |\nabla F_{j_k} (\bx^* )^{\tra} (\bx^{k}-\bx^* )|\ge |F_{j_k} (\bx^{k}) -F_{j_k} (\bx^*)|.
\end{equation}
Otherwise, one can get a contradiction with the above inequalities.

Since $\bF^{\prime}(\bx^*)$ has full column rank, by \cref{lemma:RGDC_property}, for any $\bx\in \mathcal{B}_{\rho}(\bx^*)$, 
\begin{equation*}
    \|\bF(\bx^*) - \bF(\bx)-\bF^{\prime}(\bx^*) (\bx^*-\bx)\| \le \gamma\kappa_F(\bF^{\prime}(\bx^*))\|\bF^{\prime}(\bx^* ) (\bx^*-\bx)\|.
\end{equation*}
Because of $\bx^k\in\mathcal{B}_{\rho}(\bx^*)$, by triangle inequality, we can conclude that
\begin{align*}
    \|\bF(\bx^*) - \bF(\bx^k)\| \ge 
    \left(1-\gamma\kappa_F(\bF^{\prime}(\bx^*))\right) \|\bF^{\prime}(\bx^* ) (\bx^*-\bx^k)\|.
\end{align*}   
Combining with \cref{eq:theta_res} and the above relationship, we obtain
\begin{multline}\label{eq:res_esti_dist_xk}
    |F_{j_k} (\bx^{k}) - F_{j_k} (\bx^*)| \ge \theta\|\bF(\bx^{k}) - \bF(\bx^*)\|  \\  \ge \theta \left(1-\gamma\kappa_F(\bF^{\prime}(\bx^*))\right) \|\bF^{\prime}(\bx^* ) (\bx^{k}-\bx^*)\|  \\
    \ge \theta \left(1-\gamma\kappa_F(\bF^{\prime}(\bx^*))\right)\sigma_{\min}(\bF^{\prime}(\bx^* ) )
    \|\bx^{k}-\bx^*\|.
\end{multline}
The last inequality comes from that $\bF^{\prime}(\bx^*)$ has full column rank.
Note that $\gamma$ satisfies \cref{eq:condi_gamma_le}, then $\gamma\kappa_F(\bF^{\prime}(\bx^*))<1$. Thus, \cref{eq:res_esti_dist_xk} is meaningful. More precisely, we can rewrite \cref{eq:condi_gamma_le} as 
\begin{equation*}
    \gamma \kappa_F(\bF^{\prime}(\bx^*))\le \frac{\theta(1-\tau)}{2}
    \left(1-\gamma\kappa_F(\bF^{\prime}(\bx^*))\right).
\end{equation*}
Hence, using \cref{eq:estimate_part1}, the property of Frobenius norm, the above inequality, \cref{eq:res_esti_dist_xk} and \cref{eq:estimate_part2} in order, we have
\begin{multline*}
    |\textrm{Part}\ 1| \le \gamma 
    \|\nabla F_{j_k} (\bx^* )\| \|\bx^{k}-\bx^*\|
    \le \gamma \|\bF^{\prime} (\bx^* )\|_F \|\bx^{k}-\bx^*\| \\
    \le \frac{\theta(1-\tau)}{2}
    \left(1-\gamma\kappa_F(\bF^{\prime}(\bx^*))\right)\sigma_{\min}(\bF^{\prime}(\bx^* ) ) \|\bx^{k}-\bx^*\| \\
    \le \frac{1-\tau}{2} |F_{j_k} (\bx^{k}) -F_{j_k} (\bx^*)|
    \le \frac{1-\tau}{2} |\textrm{Part}\ 2|.
\end{multline*}

Using the fact that $\textrm{Part}\ 1\ge0$ (see \cref{eq:estimate_part1}) and $\textrm{Part}\ 2<0$ (see \cref{eq:sign_part2}), we obtain 
\begin{equation*}
    \nabla F_{j_k} (\bx^{k} )^{\tra} (\bx^{k}-\bx^* ) = |\textrm{Part}\ 1| - |\textrm{Part}\ 2|
    \le -\frac{1+\tau}{2} |\textrm{Part}\ 2|<0.
\end{equation*}
Consequently, by \cref{eq:estimate_part2}, it follows that
\begin{equation*}
    |\nabla F_{j_k} (\bx^{k} )^{\tra} (\bx^{k}-\bx^* )|
    \ge \frac{1+\tau}{2} |\textrm{Part}\ 2|\ge \frac{1+\tau}{2} |F_{j_k} (\bx^{k}) -F_{j_k} (\bx^*)|.
\end{equation*}
The above two inequalities establish the two statements, respectively.

Since $\bx^0\in\mathcal{B}_{\rho}(\bx^*)$, the monotonic decrease of the distance between $\bx^k$ and $\bx^*$ holds for $k\ge0$. Thus $\bx^k$ is always in $\mathcal{B}_{\rho}(\bx^*)$.
Then, by \cref{lemma:RGDC_bound},
\begin{equation*}
    \| \bF^{\prime}(\bx^k)\|_F^2 = \sum_{j=1}^J \|\nabla F_j(\bx^k)\|^2  \le \sum_{j=1}^J (1+\gamma)^2\|\nabla F_j(\bx^*)\|^2  =  (1+\gamma)^2\|\bF^{\prime}(\bx^*)\|_F^2.
\end{equation*}
Hence, using \cref{eq:theta_res}, the equivalence of matrix norm, the above inequality and \cref{lemma:RGDC_property} in order, we obtain
\begin{multline*}
    \frac{|F_{j_k}(\bx^{k})-F_{j_k}(\bx^*)|^2  }{\|\nabla F_{j_k}(\bx^{k})\|^2}
    \ge \frac{\theta^2\|\bF(\bx^k)-\bF(\bx^*)\|^2}{\|\nabla F_{j_k}(\bx^{k})\|^2}\\
    \ge \frac{\theta^2\|\bF(\bx^k)-\bF(\bx^*)\|^2}{\|\bF^{\prime}(\bx^{k})\|_F^2}
    \ge \frac{\theta^2\left(1-\gamma\kappa_F(\bF^{\prime}(\bx^*))\right)^2 \|\bF^{\prime}(\bx^* ) (\bx^k-\bx^*)\|^2}{(1+\gamma)^2\|\bF^{\prime}(\bx^*)\|_F^2} \\
    \ge \frac{\theta^2\left(1-\gamma\kappa_F(\bF^{\prime}(\bx^*))\right)^2 \sigma_{\min}(\bF^{\prime}(\bx^*))^2 \|\bx^k-\bx^*\|^2}{(1+\gamma)^2\|\bF^{\prime}(\bx^*)\|_F^2} \\
    = \frac{\theta^2\left(1-\gamma\kappa_F(\bF^{\prime}(\bx^*))\right)^2  \|\bx^k-\bx^*\|^2}{(1+\gamma)^2 \kappa_F^2(\bF^{\prime}(\bx^*))}.
\end{multline*}
When $F_{j_k}(\bx^{k})>0$, by \cref{eq:decrease_dist_ge0} and the above estimation, we have 
\begin{align*}
    \frac{\|\bx^{k+1} -\bx^*\|^2}{\|\bx^{k} -\bx^*\|^2}  \le 1 - \frac{|F_{j_k}(\bx^{k})-F_{j_k}(\bx^*)|^2  }{\|\nabla F_{j_k}(\bx^{k})\|^2 \|\bx^{k} -\bx^*\|^2} 
    \le 1- \frac{\theta^2\left(1-\gamma\kappa_F(\bF^{\prime}(\bx^*))\right)^2 }{(1+\gamma)^2 \kappa_F^2(\bF^{\prime}(\bx^*))} < 1.
\end{align*}
For the other case of $F_{j_k}(\bx^{k})<0$, by the proven two statements, we have
\begin{align*}
    \frac{\|\bx^{k+1} -\bx^*\|^2}{\|\bx^{k} -\bx^*\|^2} \le 1 - \frac{\tau|F_{j_k}(\bx^{k})-F_{j_k}(\bx^*)|^2  }{\|\nabla F_{j_k}(\bx^{k})\|^2 \|\bx^{k} -\bx^*\|^2}
    \le 1- \frac{\tau\theta^2\left(1-\gamma\kappa_F(\bF^{\prime}(\bx^*))\right)^2 }{(1+\gamma)^2 \kappa_F^2(\bF^{\prime}(\bx^*))} < 1,
\end{align*}
which completes the proof.
\end{proof}

Let us rethink the above proof from a geometric perspective. When $F_{j_k}(\bx^{k})>0$,
using the convexity of $F_{j_k}$, we immediately conclude that 
\begin{equation*}
    \langle\bx^{k+1}-\bx^{k}, \bx^*-\bx^{k} \rangle  = \frac{F_{j_k}(\bx^{k})}{\|\nabla F_{j_k}(\bx^{k})\|^2} \nabla F_{j_k} (\bx^{k} )^{\tra} (\bx^{k}-\bx^* ) \ge \frac{F_{j_k}(\bx^{k})^2}{\|\nabla F_{j_k}(\bx^{k})\|^2}\ge 0.
\end{equation*}
Note that the nonnegativity of the above formula is necessary to ensure that $\bx^{k+1}$ is a better approximation to the solution $\bx^*$ than $\bx^k$, as shown in the left subfigure of  \cref{fig:2D_diagram}. 

On the other hand, for the case of $F_{j_k}(\bx^{k})<0$, the nonnegativity of the above formula cannot be determined only by the convexity. Fortunately, with the additional RGDC, we can prove statement (i) then further obtain that nonnegativity, as shown in the right subfigure of \cref{fig:2D_diagram}.
\begin{figure}[htbp]
    \centering
    \includegraphics[width=0.49\textwidth]{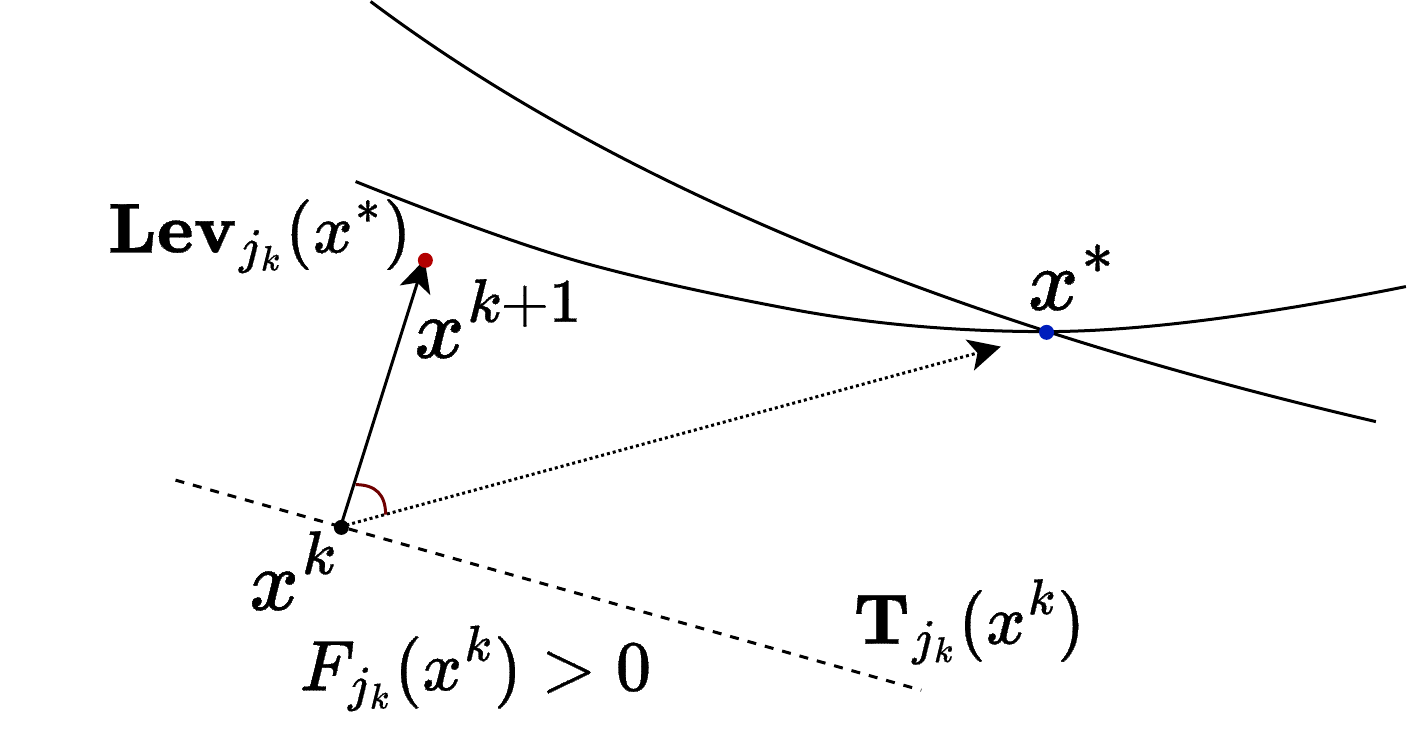}
    \includegraphics[width=0.49\textwidth]{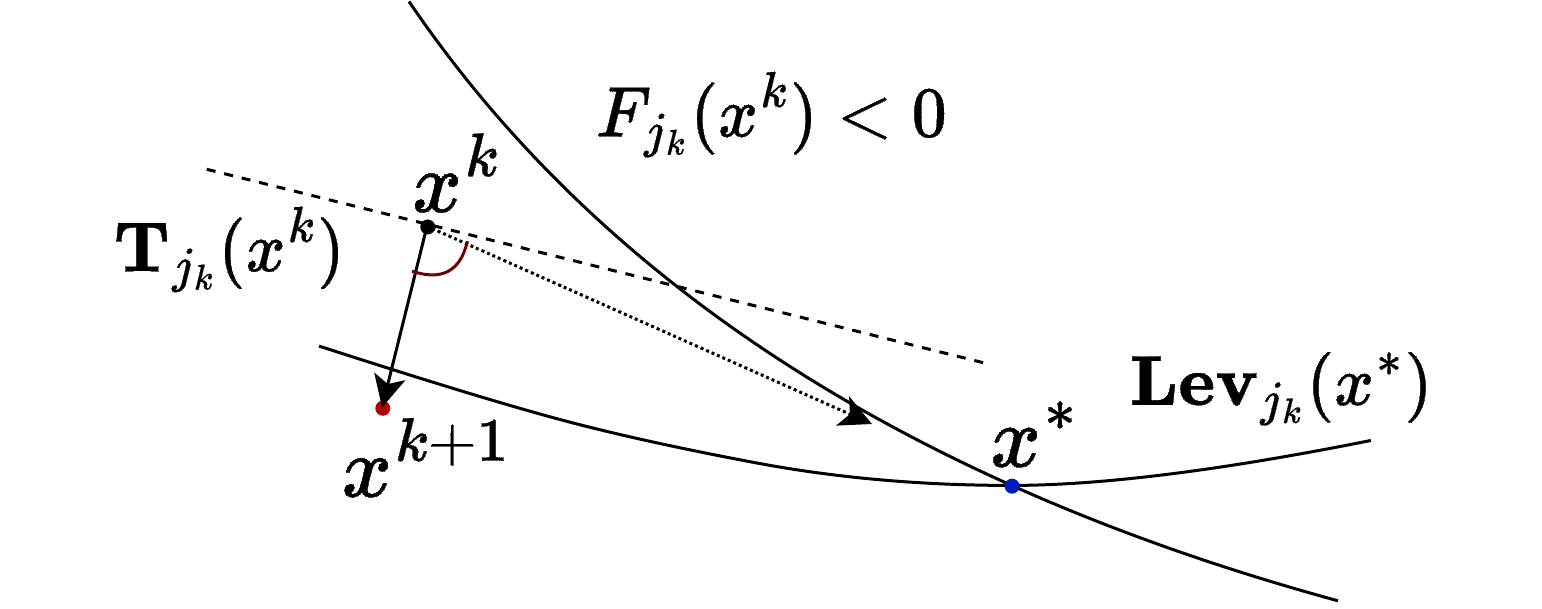}
        \vspace{3mm}
    \caption{The figure depicts the two cases for the sign of $F_{j_k}(\bx^k)$ in the NKM iteration \cref{eq:non_kacz}, as analyzed in \cref{thm:nkm_converge_rate}.}
    \label{fig:2D_diagram}
\end{figure}

\section{Convergence of the NKM for MSCT reconstruction}
\label{sec:NKM_conver_MSCT}
In this section, we specifically investigate the convergence of the NKM when it is applied to solve the nonlinear system in MSCT reconstruction. 

\subsection{Discrete-to-discrete (DD)-data model of MSCT}
The MSCT, for example, dual-energy CT (DECT), represents an emerging field of clinical CT imaging that enables the differentiation of materials with different effective atomic numbers \cite{alma76,DE_MECT15,chpan17}. The MSCT aims to derive the spatial distribution of linear attenuation coefficients, also known as a virtual monochromatic image (VMI), at a particular X-ray energy level within the imaged subject. Typically, a VMI is expressed as a linear combination of basis images, each with known decomposition coefficients. Consequently, reconstructing a VMI is equivalent to reconstructing the basis images.

In practice, the data are measured from scanning an object only by discrete rays and energies, and an image is reconstructed generally on an array of discrete pixels/voxels. Hence, the measurement without other physical factors can be modeled as the formula below 
\begin{equation}\label{eq:dd_formula}
g_j^{[p]}  = \ln\sum_{m=1}^{M} s^{[p]}_{m} \exp\left(-\sum_{d=1}^D b_{dm} z^{[p]}_{j d}
\right) \quad \text{for $1 \le j \le J_p$, $1\le p\le P$,}
\end{equation}
where $\sum_{m=1}^{M} s^{[p]}_{m} =1$, and 
\begin{equation}\label{eq:linear_part}
z^{[p]}_{j d} = \sum_{i=1}^I a^{[p]}_{ji} f_{di} = {\bda_j^{[p]}}^{\tra} \bdf_{\!d}.
\end{equation}
Here $\bda_j^{[p]} = [a^{[p]}_{j1},\ldots,a^{[p]}_{jI}]^{\tra}$, $\bdf_{\!d} = [f_{d1},\ldots,f_{dI}]^{\tra}$.
For clarity, the physical meaning of the used notation is listed in \cref{tab:notation}. Note that these physical terms except measurement $g_j^{[p]}$ are all nonnegative. The VMI at the $m$-th energy bin can also be discretized accordingly as 
\begin{equation*}
    \boldsymbol{\mu}_m:=\sum_{d=1}^D b_{dm} \bdf_{\!d}.
\end{equation*}

\begin{table}
    \centering
    \caption{The physical meaning for terms in DD-data model \cref{eq:dd_formula}--\cref{eq:linear_part}.}
    \label{tab:notation}
    \begin{tabular}{ll}
        \toprule
        Notation &  Description \\
        \midrule
        $M$ & Number of energy bins \\[1mm]
        $P$ & Number of energy spectra \\[1mm]
        $D$ & Number of basis decompositions \\[1mm]
        $I$ & Size of discrete basis image array in a concatenated form \\[1mm]
        $J_p$ & Total number of rays with the $p$-th energy spectrum \\[1mm]
        $g_j^{[p]}$ & Measured data along the $j$-th ray with the $p$-th energy spectrum\\[1mm]
        $s^{[p]}_m$ & The $p$-th discrete normalized spectrum at energy bin $m$  \\[1mm]
        $b_{dm}$ & Decomposition coefficient at energy bin $m$ of the $d$-th basis  \\[1mm]
        $z_{jd}^{[p]}$ & Basis sinogram of ray $j$ with the $p$-th spectrum for basis image $d$  \\[1mm]
        $a_{ji}^{[p]}$ & Intersection length of ray $j$ with the $p$-th spectrum for voxel $i$ \\[1mm]
        $f_{di}$ & Voxel $i$ of the $d$-th basis image \\[1mm]
        $\bda_j^{[p]}$ & Discrete X-ray transform along ray $j$ with the $p$-th spectrum \\[1mm]
        $\bdf_{\!d}$ & Discrete array of the $d$-th basis image \\[1mm]
        \bottomrule
    \end{tabular}
\end{table}

Due to the different data acquisition methods, the configuration of the scanning geometry in MSCT can be generally divided into two types, including geometrically-consistent and geometrically-inconsistent. Thus, the image reconstruction problem in MSCT can be transformed into two types of nonlinear systems according to the above two different configurations. In what follows, we use the theoretical results in \cref{sec:conver_analy} to analyze the convergence of the NKM when solving these two nonlinear systems.

\subsection{Geometrically-consistent MSCT}\label{subsec:geo_con_analysis}

Here we prove a global convergence of the NKM when it is applied to solve the nonlinear system in geometrically-consistent MSCT reconstruction.

In geometrically-consistent MSCT, the configurations of the scanning geometries are consistent for all spectra  \cite{Zou08,Bal_2020,gao23solution}. This consistency corresponds to that $a_{ji}^{[p]}$ and $J_p$ are independent of spectra, \ie, are identical for all $1\le p\le P$, and thus can be denoted as $a_{ji}$ and $J$, respectively. This allows us to rewrite \cref{eq:dd_formula} and \cref{eq:linear_part} as 
\begin{equation}\label{eq:dd_formula_GC}
    g_j^{[p]}  = \ln\sum_{m=1}^{M} s^{[p]}_{m} \exp\left(-\sum_{d=1}^D b_{dm} z_{j d}
    \right) \quad \text{for $1 \le j \le J$, } 1\le p\le P,
\end{equation}
and 
\begin{equation}\label{eq:linear_part_GC}
    z_{j d} = \sum_{i=1}^I a_{ji} f_{di} = {\bda_j}^{\tra} \bdf_{\!d}.
\end{equation}
Moreover, the measured data and basis sinograms from $P$ different spectra can be organized as the following $J$ independent nonlinear systems
\begin{equation}\label{eq:GC_nonsys}
    \bH(\bz_j) - \bdg_j = \bzero \quad\text{for}\ j =1,\ldots, J,
\end{equation}
where $\bdg_j= [g_{j}^{[1]}, \ldots, g_j^{[P]}]^{\tra}$ and $\bz_j=[z_{j1}, \ldots, z_{jD}]^{\tra}$ denote the corresponding measured data and basis sinograms along ray $j$, respectively, $\bH$ is defined as $\bH(\bz) = [H_1(\bz), \ldots, H_{P}(\bz)]^{\tra}$ and 
\begin{align}\label{eq:H_def}
    H_p(\bz) := \ln\sum_{m=1}^{M} s_{m}^{[p]} \exp\left(-\sum_{d=1}^D b_{dm} z_{d}\right).
\end{align}
In the widely used two-step data-domain-decomposition (DDD) method, step (1) solves numerous small-scale nonlinear systems in \cref{eq:GC_nonsys} to obtain
the basis sinograms $\{\bz_j\}$, and step (2) inverts a linear system \cref{eq:linear_part_GC} to reconstruct the basis images $\{\bdf_{\!d}\}$ from the obtained basis sinograms.

We re-express the small-scale nonlinear system \cref{eq:GC_nonsys} in the form of \cref{eq:nonlinear_system} as
\begin{equation}\label{eq:nonsys_GC_nkm}
    \bH(\bz) - \bdg = \bzero \Longleftrightarrow
    H_p(\bz) - g^{[p]} = 0 \quad \text{for}\ p = 1,\ldots,P.
\end{equation}
The NKM for solving \cref{eq:nonsys_GC_nkm} is presented in \cref{alg:nkm_GC}.
Using the auxiliary results in \cref{subsec:verify_convex,subsec:verify_RGDC}, together with the previous result in \cref{thm:nkm_converge_rate}, we immediately obtain a global convergence of the NKM for solving \cref{eq:nonsys_GC_nkm}.

\begin{algorithm}[h]
    \caption{The NKM for solving \cref{eq:nonsys_GC_nkm}}
    \label{alg:nkm_GC}
    \begin{algorithmic}[1]
   \STATE \emph{Initialization}: Given $S$, $B$, the data $\bdg$, and the initial point $\bz^0$. Let $k\gets 0$.
    \STATE \emph{Loop:}
    \STATE \quad Select an index $p_k\in\{1,\ldots,P\}$ according to some strategy.
    \STATE \quad Compute $H_{p_k}(\bz^k)$ and $\bdw^{[p_k]}(\bz^k)$ by \cref{eq:H_def,eq:def_wx}, respectively.
    \STATE \quad Update $\bz^{k+1}$ by 
    \begin{equation*}
        \bz^{k+1} = \bz^{k} + \frac{H_{p_k}(\bz^k) - g^{[p_k]} }{\|B \bdw^{[p_k]}(\bz^k)\|^2} B \bdw^{[p_k]}(\bz^k).
    \end{equation*}
    \STATE \quad \textbf{If} some given termination condition is satisfied, \textbf{output} $\bz^{k+1}$; 
    \STATE \quad \textbf{otherwise}, let $k \gets k+1$, \textbf{goto} \emph{Loop}.
    \end{algorithmic}
\end{algorithm}

\begin{theorem}\label{thm:kac_GC_converge}
    Let $\{\bz^{k}\}$ be the sequence generated by \cref{alg:nkm_GC} using strategy \ref{stra:theta_res} for $\theta \in (0, 1/\sqrt{J}]$. Suppose that \cref{eq:nonsys_GC_nkm} exists a solution $\bz^*$ and for any index set $\alpha\subseteq \{1,\ldots,M\}$,  the number of its elements $\#\alpha = P=D$ such that 
    \begin{equation}\label{eq:det_same_sign}
        \det(S[\{1,\ldots,P\},\alpha]) 
        \det(B[\{1,\ldots,P\},\alpha]) \ge 0 \ (\text{or}\ \le0),
    \end{equation}
    where $S = (s_{m}^{[p]}) \in \Real^{{P}\times M}, B = (b_{dm}) \in \Real^{D\times M}$. Assume further that the $\gamma_B$ defined in \cref{eq:gamma_general} satisfies 
    \begin{equation*}
        \gamma_B \kappa_F(\bH^{\prime}(\bz^*))\le \frac{\theta(1-\tau)}{ 2+\theta(1-\tau)}
    \end{equation*}
    for $0<\tau<1$. 
    Then for any given initial point, the sequence converges to the unique solution $\bz^*$ as $k\rightarrow\infty$ at least with a linear rate.
\end{theorem}
\begin{proof}
By the conclusion of theorem 4.1 in \cite{gao23solution}, we know that $\det(\bH^{\prime}(\bz))$ never vanishes for all $\bz\in\Real^D$. Thus, $\bH^{\prime}(\bz^*)$ has full column rank. Using the verifications in \cref{lem:map_MSCT_convex,lem:general_verify_grad}, it follows that $\bH$ satisfies the conditions of \cref{lemma:unique_solution,lem:dist_decrease} in $\mathcal{B}_{\rho}(\bz^*)$ with any radius $\rho$. Hence, by \cref{thm:nkm_converge_rate}, the sequence converges to the unique solution $\bz^*$ in $\Real^D$. Moreover, its convergence rate is at least linear as 
\begin{align*}
    \frac{\|\bz^{k+1} -\bz^*\|}{\|\bz^{k} -\bz^*\|} &\le \sqrt{1- \frac{\tau\theta^2\left(1-\gamma_B\kappa_F(\bH^{\prime}(\bz^*))\right)^2 }{(1+\gamma_B)^2 \kappa_F^2(\bH^{\prime}(\bz^*))}}.
\end{align*}
\end{proof}

\subsection{Geometrically-inconsistent MSCT}\label{subsec:geo_incon_analysis}

In this subsection, we prove a global convergence of the NKM when it is applied to solve the nonlinear system in geometrically-inconsistent MSCT image reconstruction.

For the geometrically-inconsistent case, it is usually necessary to directly solve \cref{eq:dd_formula} to
reconstruct the basis images $\{\bdf_{\!d}\}$. We put all equations in \cref{eq:dd_formula,eq:linear_part} together and express in the form of \cref{eq:nonlinear_system} as
\begin{equation}\label{eq:nonsys_GIC_nkm}
    \bdK(\mathbf{f}) - \mathbf{g} = \bzero \Longleftrightarrow
    K_{j}^{[p]} (\mathbf{f}) - g_{j}^{[p]} = 0  \quad \text{for}\ j = 1, \ldots, J_p, 1\le p\le P,
\end{equation}
where $\bdK(\mathbf{f}) = [K_1^{[1]}(\mathbf{f}),\ldots, K^{[P]}_{J_P}(\mathbf{f})]^{\tra}$, $\mathbf{g} = [g_1^{[1]},\ldots, g_{J^{[P]}}^{[P]}]^{\tra}$ and $\mathbf{f}= [\bdf_{\!1}^{\tra},\ldots,\bdf_{\!D}^{\tra}]^{\tra}$,
\begin{equation}\label{eq:log_sum_exp}
K_j^{[p]}(\mathbf{f}) := \ln\sum_{m=1}^{M} s_{m}^{[p]} \exp\Biggl(-\sum_{d=1}^{D} b_{dm}{\bda_j^{[p]}}^{\tra} \bdf_{\!d} \Biggr).
\end{equation}
Let $\bz_j^{[p]} = [{\bda_j^{[p]}}^{\tra} \bdf_{\!1},\ldots,{\bda_j^{[p]}}^{\tra} \bdf_{\!D}]^{\tra}$, then $K_j^{[p]}(\mathbf{f})=H_p(\bz_j^{[p]})$.
Consequently, the 
gradient of $K_j^{[p]}(\mathbf{f})$ is given by 
\begin{equation*}
    \nabla K_j^{[p]}(\mathbf{f}) = \nabla H_p(\bz_j^{[p]}) \otimes\bda_j^{[p]},
\end{equation*}
where $\otimes$ denotes the Kronecker product. We present the NKM for solving \cref{eq:nonsys_GIC_nkm} in \cref{alg:nkm_GIC}. 
\begin{algorithm}[htbp]
    \caption{The NKM for solving \cref{eq:nonsys_GIC_nkm}}
    \label{alg:nkm_GIC}
    \begin{algorithmic}[1]
    \STATE \emph{Initialization}: Given $S,B$, $\{\bda_j^{[p]}\}$, the data $\mathbf{g}$, and the initial point $\mathbf{f}^0$. Let $k\gets 0$. 
    \STATE \emph{Loop:}
    \STATE \quad Select indices $p_k\in\{1,\ldots,P\}$ and $j_k\in\{1,\ldots,J_{p_k}\}$ according to \\
    \quad some strategy.
    \STATE \quad Compute the forward projections $\bz_{j_k}^{[p_k]} = [{\bda_{j_k}^{[p_k]}}^{\tra} \bdf_{\!1}^k,\ldots,{\bda_{j_k}^{[p_k]}}^{\tra} \bdf_{\!D}^k]^{\tra} $, \\ \quad $\bdw^{[p_k]}(\bz_{j_k}^{[p_k]})$ and $K_{j_k}^{[p_k]}(\mathbf{f}^k)$ by \cref{eq:log_sum_exp,eq:def_wx}, respectively. 
    \STATE \quad Update $\mathbf{f}^{k+1}$ by 
    \begin{equation*}
        \mathbf{f}^{k+1} = \mathbf{f}^{k} + \frac{ K_{j_k}^{[p_k]}(\mathbf{f}^k) - g_{j_k}^{[p_k]}}{\|B \bdw^{[p_k]}(\bz_{j_k}^{[p_k]})\|^2 \|\bda_{j_k}^{[p_k]}\|^2} B \bdw^{[p_k]}(\bz_{j_k}^{[p_k]})\otimes\bda_{j_k}^{[p_k]}.
    \end{equation*}
    \STATE \quad \textbf{If} some given termination condition is satisfied, \textbf{output} $\mathbf{f}^{k+1}$; 
    \STATE \quad \textbf{otherwise}, let $k \gets k+1$, \textbf{goto} \emph{Loop}.
    \end{algorithmic}
\end{algorithm}

\begin{theorem}\label{thm:kac_GIC_converge}
    Let $\{\mathbf{f}^k\}$ be the sequence generated by \cref{alg:nkm_GIC} using strategy \ref{stra:theta_res} for $\theta \in (0, 1/\sqrt{J}]$.
    Suppose that \cref{eq:nonsys_GIC_nkm} exists a solution $\mathbf{f}^*$, and $\bdK^{\prime}(\mathbf{f}^*)$ has full column rank, the $\gamma_B$ defined in \cref{eq:gamma_general} satisfies 
    \begin{equation*}
        \gamma_B \kappa_F(\bdK^{\prime}(\mathbf{f}^*))\le \frac{\theta(1-\tau)}{ 2+\theta(1-\tau)}
    \end{equation*}
    for $0<\tau<1$. Then for any given initial point, the sequence converges to the unique solution $\mathbf{f}^*$ as $k\rightarrow\infty$ at least with a linear rate.
\end{theorem}
\begin{proof}
    The Hessian matrix of $K_j^{[p]}(\mathbf{f})$ is given by 
    \begin{equation*}
        \nabla^2 K_j^{[p]}(\mathbf{f}) = \nabla^2  H_p(\bz_j^{[p]}) \otimes \left(\bda_j^{[p]} (\bda_j^{[p]})^{\tra} \right).
    \end{equation*}
    The eigenvalues of the above Kronecker product of two square matrices are given by $\{\lambda_d\cdot\tilde{\lambda}_i\}$ (see \cite{horn_johnson_book91}), where $\lambda_d\in \lambda\left(\nabla^2  H_p(\bz_j^{[p]})\right)$ for $d=1,\ldots D$ and $\tilde{\lambda}_i\in \lambda\left(\bda_j^{[p]} (\bda_j^{[p]})^{\tra} \right)$ for $i=1,\ldots I$.
    Recalling the proof in \cref{lem:map_MSCT_convex}, we know
    $\nabla^2  H_p(\bz)\succeq 0$ for all $\bz$. The rank-one matrix $\bda_j^{[p]} (\bda_j^{[p]})^{\tra}$ is also positive semidefinite.
    Hence, all eigenvalues of the Hessian matrix of $K_j^{[p]}(\mathbf{f})$ are nonnegative and thus $K_j^{[p]}(\mathbf{f})$ is convex, which was also proved by the other method in \cite{gao2021EPD}. Moreover, by the property of Kronecker product, we have
\begin{align*}
    \frac{\|\nabla K_j^{[p]}(\mathbf{f}) - \nabla K_j^{[p]}(\tilde{\mathbf{f}}) \| }{\|\nabla K_j^{[p]}(\mathbf{f})\|} &= \frac{\|\nabla H_p(\bz_j^{[p]}) - \nabla H_p(\tilde{\bz}_j^{[p]})\|\cdot \|\bda_j^{[p]}\|}{\|\nabla H_p(\bz_j^{[p]})\| \cdot \|\bda_j^{[p]}\|} \\
    &= \frac{\|\nabla H_p(\bz_j^{[p]}) - \nabla H_p(\tilde{\bz}_j^{[p]})\| }{\|\nabla H_p(\bz_j^{[p]})\|},
\end{align*}
where $\tilde{\bz}_j^{[p]} = [{\bda_j^{[p]}}^{\tra} \tilde{\bdf}_{\!1},\ldots,{\bda_j^{[p]}}^{\tra} \tilde{\bdf}_{\!D}]^{\tra}$.
Therefore, by \cref{lem:general_verify_grad}, we can also verify that $\bdK$ satisfies the \ref{ass:relative_grad_discrepancy}. Using \cref{lemma:unique_solution,thm:nkm_converge_rate}, we obtain the uniqueness of solution and the linear convergence.
\end{proof}

\section{Numerical experiments}
\label{sec:Numerical}

In this section, we validate the numerical convergence of the NKM when it is applied to solve the two types of nonlinear systems in \cref{eq:nonsys_GC_nkm} (or \cref{eq:GC_nonsys}) and \cref{eq:nonsys_GIC_nkm} for DECT in \cref{sec:NKM_conver_MSCT}. Here both numbers of energy spectra and basis decomposition are equal to 2, \ie, $P=D=2$.

The two tests are both performed on the workstation with Intel(R) Xeon(R) Gold 6140 CPU @2.3GHz$\times2$ and 192GB RAM running Python. The used settings for tests are presented below, which includes energy spectra, decomposition coefficients, truth basis images and discrete X-ray transform. 

    The discrete X-ray spectra $\{s_m^{[p]}\}$ in \cref{eq:nonsys_GC_nkm,eq:nonsys_GIC_nkm} are simulated by an open-source X-ray spectrum simulator, SpectrumGUI \cite{spectrum_GUI}. We employ two spectra pairs displayed in \cref{fig:spectra_pair_contrast}, which are referred to as spectra pair I and spectra pair II, respectively. The spectra pair I contains 80-kV and 140-kV spectrum of the GE maxiray 125 X-ray tube. Additionally, the spectra pair II consists of the same 80-kV spectrum as pair I and a filtered 140-kV spectrum by a copper filter of 1-mm width. 
    \begin{figure}[htbp] 
        \centering 
       \includegraphics[width=0.45\textwidth]{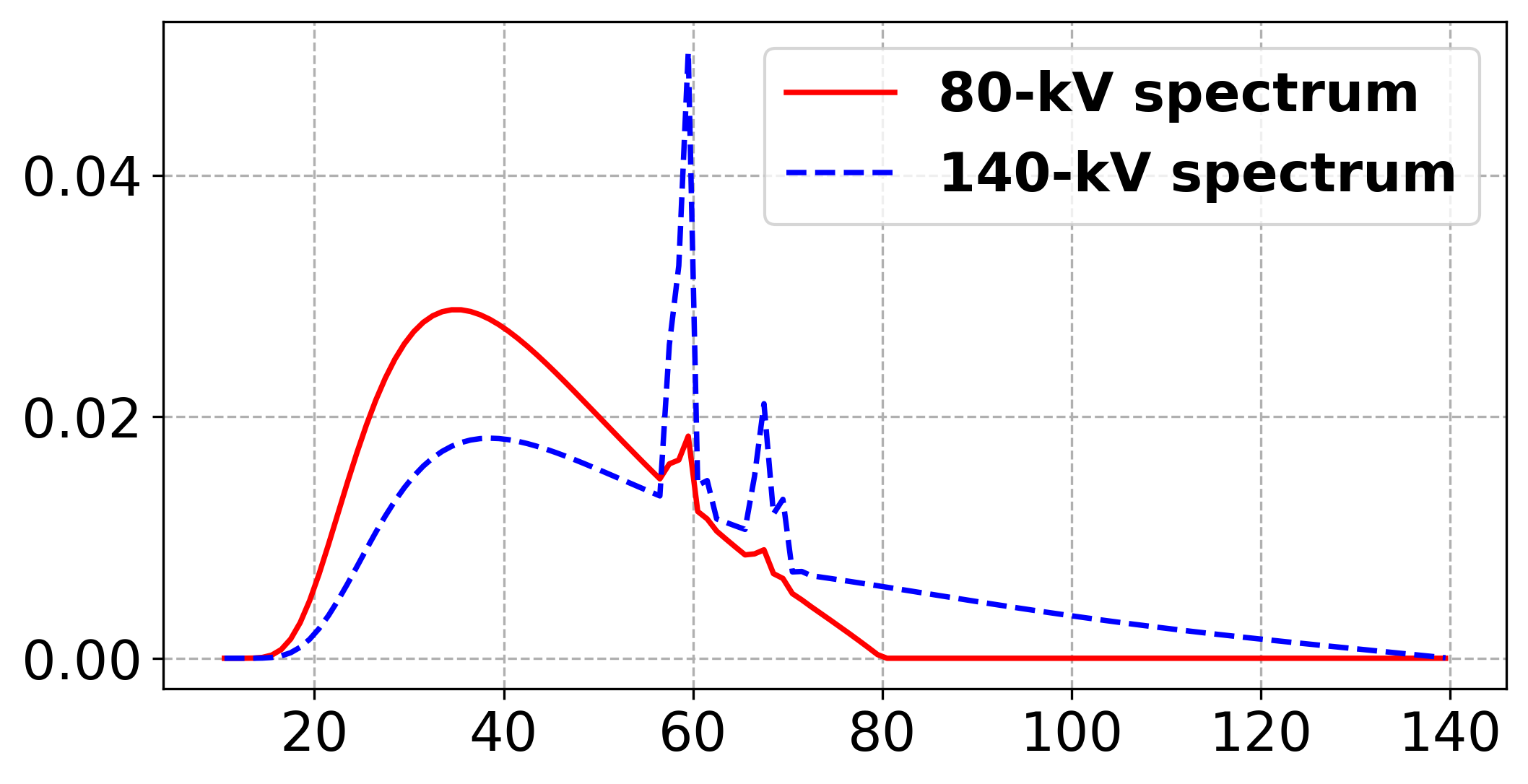}
       \includegraphics[width=0.45\textwidth]{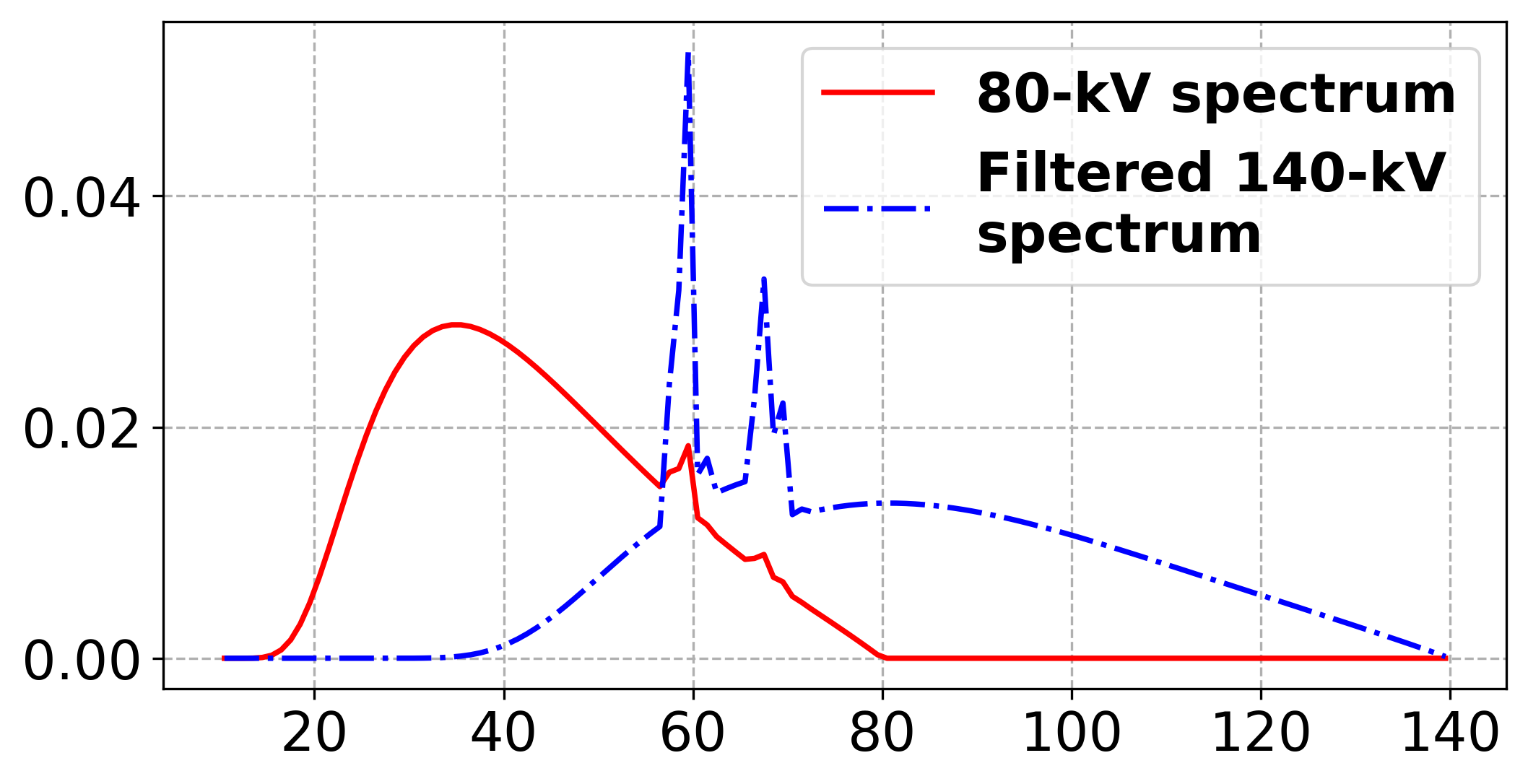} 
           \vspace{3mm}
    \caption{The used spectra pairs. Left: Spectra pair I, which contains 80-kV spectrum (red, solid) and 140-kV spectrum (blue, dashed). Right: Spectra pair II, which contains 80-kV spectrum (red, solid) and filtered 140-kV spectrum (blue, dashed dot).} 
         \label{fig:spectra_pair_contrast} 
\end{figure} 

    We take the mass attenuation coefficients of the used basis materials, namely, water and bone, from \cite{NIST_data} as the decomposition coefficients $\{b_{dm}\}$ in \cref{eq:nonsys_GC_nkm,eq:nonsys_GIC_nkm}. 
    The forbild head phantom is chosen as the truth water and bone basis images which are both discretized as $128\times128$ pixels on square area $[-1,1]\times[-1,1]$ $\text{cm}^2$ and take gray values over $[0, 1]$, as shown in row 1 of \cref{fig:recon_forbild}. The discrete X-ray transform, equivalently, the projection matrices $\{\bda_j^{[p]}\}$ are computed by the algorithm proposed in \cite{chen2020fast}. 

With the above settings, for validating the convergence of the associated NKM, we generate the simulated data without noise by plugging the used energy spectra, decomposition coefficients, discrete X-transform and truth basis images into \cref{eq:nonsys_GC_nkm} and \cref{eq:nonsys_GIC_nkm}. The stability study using noisy data for the NKM is out of the scope of this work. For clarity in the following context, one $\textit{iteration}$ refers to once updating by \cref{eq:non_kacz}, while one $\textit{epoch}$ represents completing a set of iterations equal to the total number of equations in the system.

\subsection{Test 1: Geometrically-consistent DECT}\label{test1:GC_DECT}

In this case, we focus on using the NKM (\cref{alg:nkm_GC}) to solve numerous nonlinear systems as in \cref{eq:GC_nonsys}. 
Note that the system for each ray is independent thus can be solved in parallel. 

We define the following metric of relative error for the reconstruction error of basis sinograms
\begin{align*}
    \text{RE}_{\bz}^k := \sqrt{\frac{\sum_{j=1}^J \|\bz_j^k - \bz_j^*\|^2}{\sum_{j=1}^J \|\bz_j^*\|^2}},
\end{align*}
where $\bz_j^k$ represents the basis sinogram computed for the $j$-th ray path at the $k$-th iteration of \cref{alg:nkm_GC}, and 
$\{\bz_j^*\}$ denotes the truth basis sinograms.

The consistent scanning configuration is employed for low and high energy spectra, which comprises 181 parallel  projections uniformly sampled on $[-1.5,1.5]$ cm at each of 180 views uniformly distributed over $[0,\pi)$. Then the truth basis sinograms can be obtained by plugging the discrete X-ray transform and truth basis images in \cref{eq:linear_part_GC}, as shown in column 1 of \cref{fig:basis_sinogram}. In this test, we use two different spectra pairs I and II as shown in \cref{fig:spectra_pair_contrast} but the same decomposition coefficients to formulate two different nonlinear systems \cref{eq:H_def}, which are referred to as $\bH_1$ and $\bH_2$, respectively. Note that by direct validation, both $\bH_1$ and $\bH_2$ are satisfying the  condition \cref{eq:det_same_sign} in \cref{thm:kac_GC_converge}.

We employ \cref{alg:nkm_GC} using the maximum residual strategy \cref{eq:kacz_max_res} and set zero vector as the initial point for all nonlinear systems.
After $10^4$ epochs ($2\times 10^4$ iterations) with \cref{alg:nkm_GC}, 
the obtained metric curves of RE$_{\bz}^k$ as the function of iteration number are plotted in \cref{fig:metric_linear}.
\begin{figure}[htbp]  
    \centering 
    \includegraphics[width=0.618\textwidth]{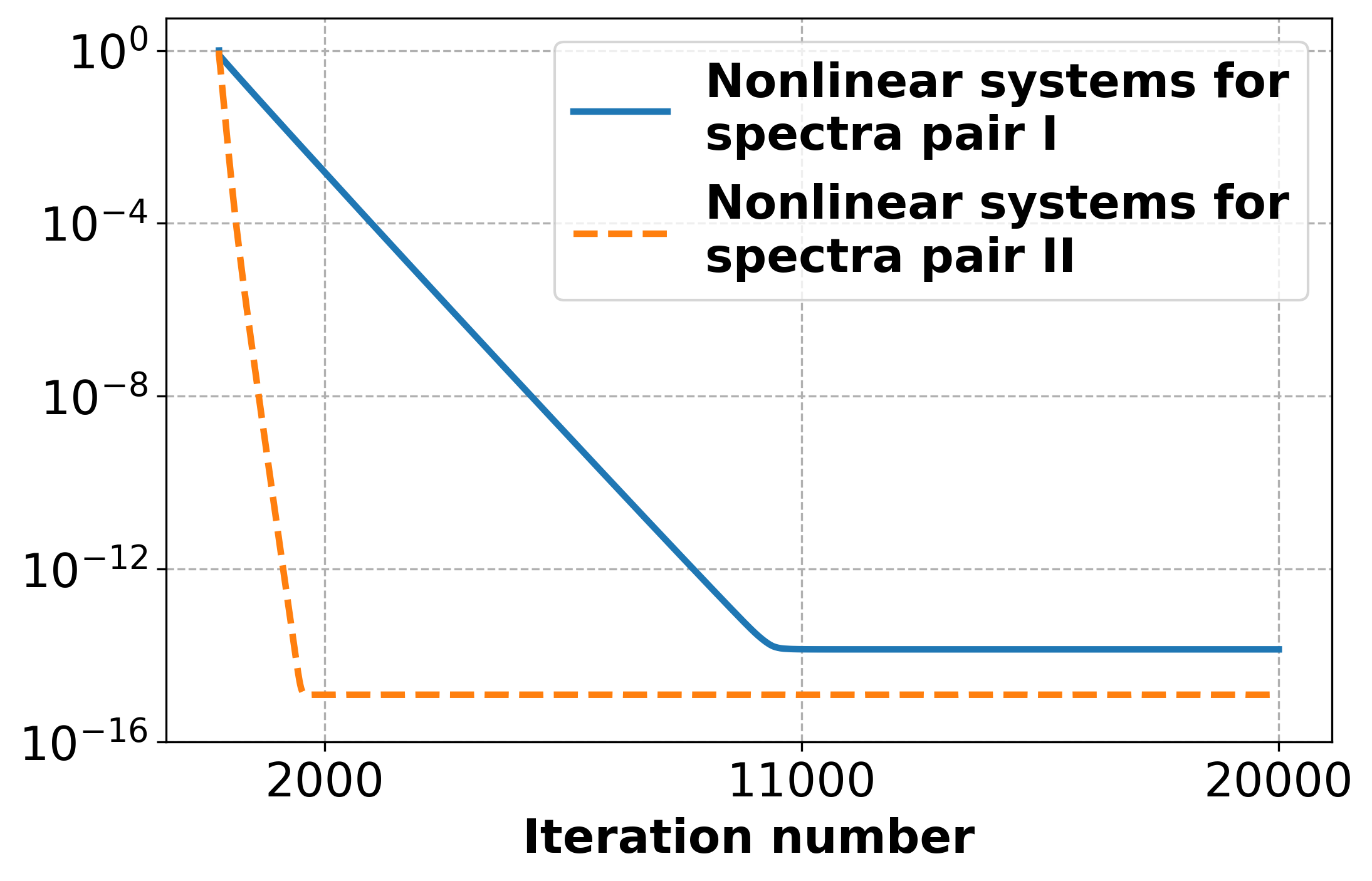}
        \vspace{3mm}
    \caption{Metric RE$_{\bz}^k$ plotted in a semi-log scale as the function of iteration number of the NKM using the maximum residual strategy. The solid and dashed curves are obtained from solving different nonlinear systems for spectral pairs in \cref{fig:spectra_pair_contrast}, respectively.} 
     \label{fig:metric_linear}
\end{figure}
It reveals that the NKM, namely, \cref{alg:nkm_GC} using the maximum residual strategy can accurately solve the nonlinear systems in \cref{eq:GC_nonsys}. The convergence rate of the NKM is linear as shown in \cref{fig:metric_linear}.

Moreover, the faster convergence is also observed in \cref{fig:metric_linear} when solving the nonlinear systems for spectra pair II than I, which is consistent with the theoretical result that the smaller the scaled condition number $\kappa_F(\bH^{\prime}(\bz^*))$, the faster the convergence rate by \cref{thm:kac_GC_converge}. 
\begin{figure}[htbp]
    \centering
    \includegraphics[width=0.95\textwidth]{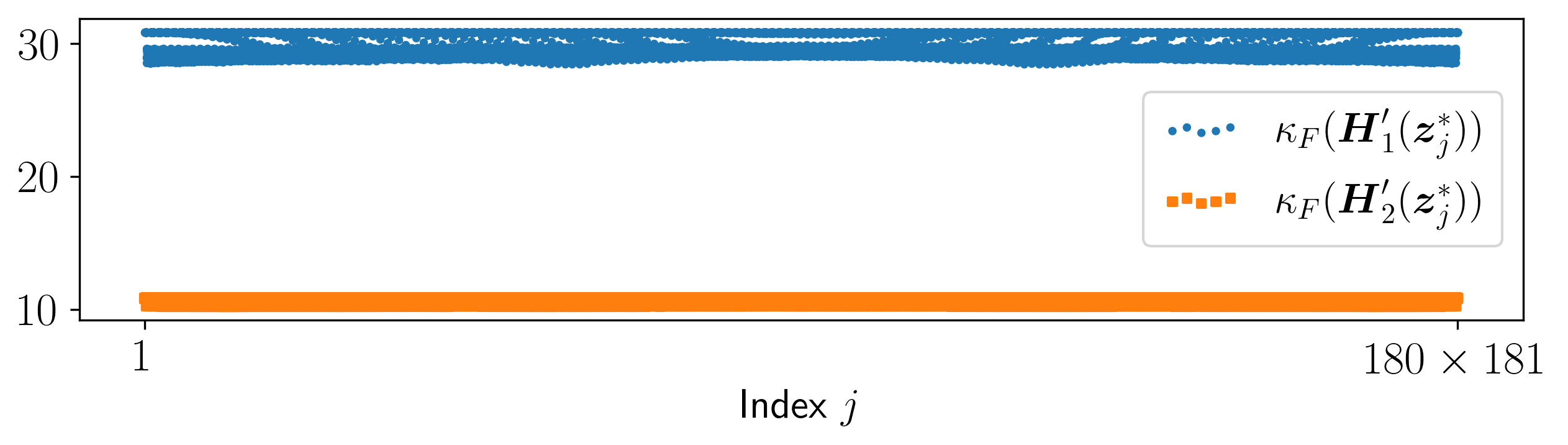}
        \vspace{3mm}
    \caption{The scatter diagram of scaled condition numbers $\kappa_F(\bH_1^{\prime}(\bz_j^*))$ (dot point) and $\kappa_F(\bH_2^{\prime}(\bz_j^*))$ (square) for every index $j$.}
    \label{fig:kappa_scatter}
\end{figure}
To illustrate this, we depict the distributions of the $\kappa_F(\bH_1^{\prime}(\bz_j^*))$ and $\kappa_F(\bH_2^{\prime}(\bz_j^*))$ for $j=1,\ldots,J$ in \cref{fig:kappa_scatter}. It can be observed that
$\kappa_F(\bH_2^{\prime}(\bz_j^*))$ is always less than $\kappa_F(\bH_1^{\prime}(\bz_j^*))$, which confirms our analysis in \cref{thm:kac_GC_converge}.

We display the results of the computed basis sinograms in \cref{fig:basis_sinogram} by using \cref{alg:nkm_GC} to solve different nonlinear systems for spectra pairs I and II. It can be found that there is no difference between the computed results and the truths visually. 
\begin{figure}[htbp]
    \centering
    \includegraphics[width=0.95\textwidth]{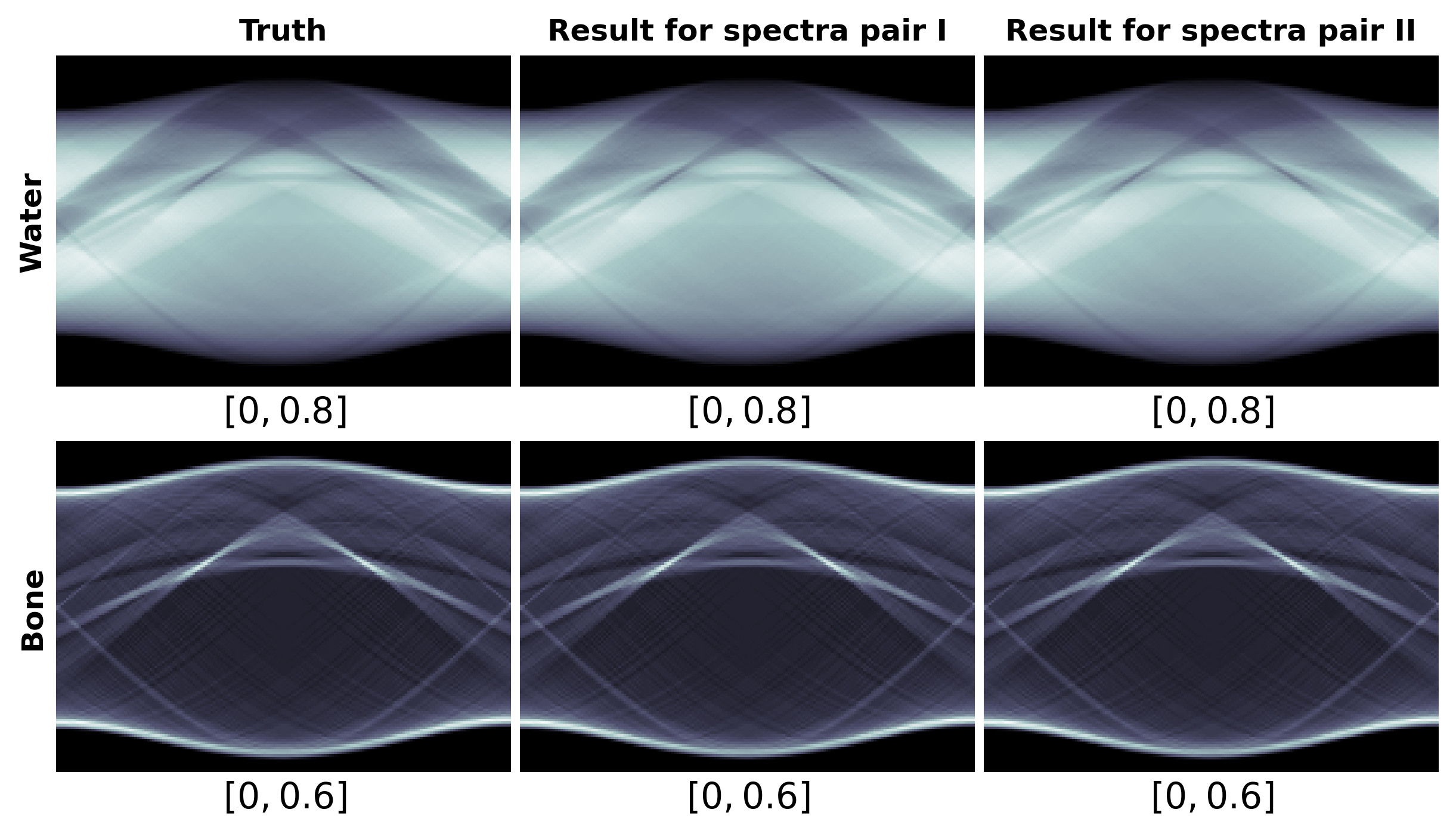}
        \vspace{3mm}
    \caption{The water and bone basis sinograms for the forbild head phantom are shown in rows 1 and 2, respectively. Column 1 displays the truth basis sinograms, while columns 2 and 3 show the computed basis sinograms obtained by the NKM after $2\times10^4$ iterations for solving different nonlinear systems corresponding to the spectra pairs presented in \cref{fig:spectra_pair_contrast}, respectively.}
    \label{fig:basis_sinogram}
\end{figure}

\subsection{Test 2: Geometrically-inconsistent DECT}
\label{test2:GIC_DECT}

In this experiment, we are concerned with the effectiveness of the NKM (\cref{alg:nkm_GIC}) using the maximum residual strategy to solve the nonlinear system \cref{eq:nonsys_GIC_nkm} in geometrically-inconsistent DECT. 
 
The following metric of relative error is used to assess the reconstruction accuracy
\begin{align}\label{eq:REk_f}
    \text{RE}_{\mathbf{f}}^k := \frac{ \|\mathbf{f}^k - \mathbf{f}^*\|}{\|\mathbf{f}^*\|},
\end{align}
where $\mathbf{f}^k$ represents the basis images computed by the NKM using the maximum residual strategy at the $k$-th iteration of \cref{alg:nkm_GIC}, and $\mathbf{f}^*$ denotes the truth basis images. The metric below is also considered to evaluate the relative error of data $\mathbf{g}$
\begin{align}\label{eq:REn_g}
    \text{RE}_{\mathbf{g}}^{(n)}:= \frac{\|\bdK(\mathbf{f}^{(n)}) - \mathbf{g}\|}{\|\mathbf{g}\|},
\end{align}
where $\mathbf{f}^{(n)}$ represents the basis images computed by the NKM using the maximum residual strategy after $n$ epochs. 

Based on the faster convergence rate observed in \cref{fig:metric_linear}, we adopt spectra pair II in this test as shown in \cref{fig:spectra_pair_contrast}. The previously forbild head phantom is also used here as the truth basis images. Additionally, we employ the geometrically-inconsistent ray paths for the different spectra. Specifically, for 80-kV spectrum, we utilize 384 parallel projections uniformly sampled on $[-1.5, 1.5]$ cm for each of the 384 views uniformly distributed over the interval $[0, \pi)$. The filtered 140-kV spectrum also has the same number of projections and views, but the views are uniformly distributed over the interval $[\pi/512, \pi/512+\pi)$. Note that the system to be solved contains $294912$ nonlinear equations and $32768$ unknowns.

Without loss of generality, we set the zero vector as the initial point of the NKM in  \cref{alg:nkm_GIC}. After 60 epochs (about $1.4\times10^7$ iterations), the curves of metrics RE$_{\mathbf{f}}^k$ and RE$_{\mathbf{g}}^{(n)}$ are plotted in \cref{fig:REk_GIC}, which reveal that under given conditions, the large-scale nonlinear system in \cref{eq:nonsys_GIC_nkm} for geometrically-inconsistent DECT can be solved accurately by the NKM (\cref{alg:nkm_GIC}) using the maximum residual strategy. As shown in figure \ref{fig:REk_GIC}, the linear convergence also confirms the theory proved in \cref{sec:conver_analy} and \cref{sec:NKM_conver_MSCT}. 
\begin{figure}[htbp]
    \centering
    \includegraphics[width=0.45\textwidth]{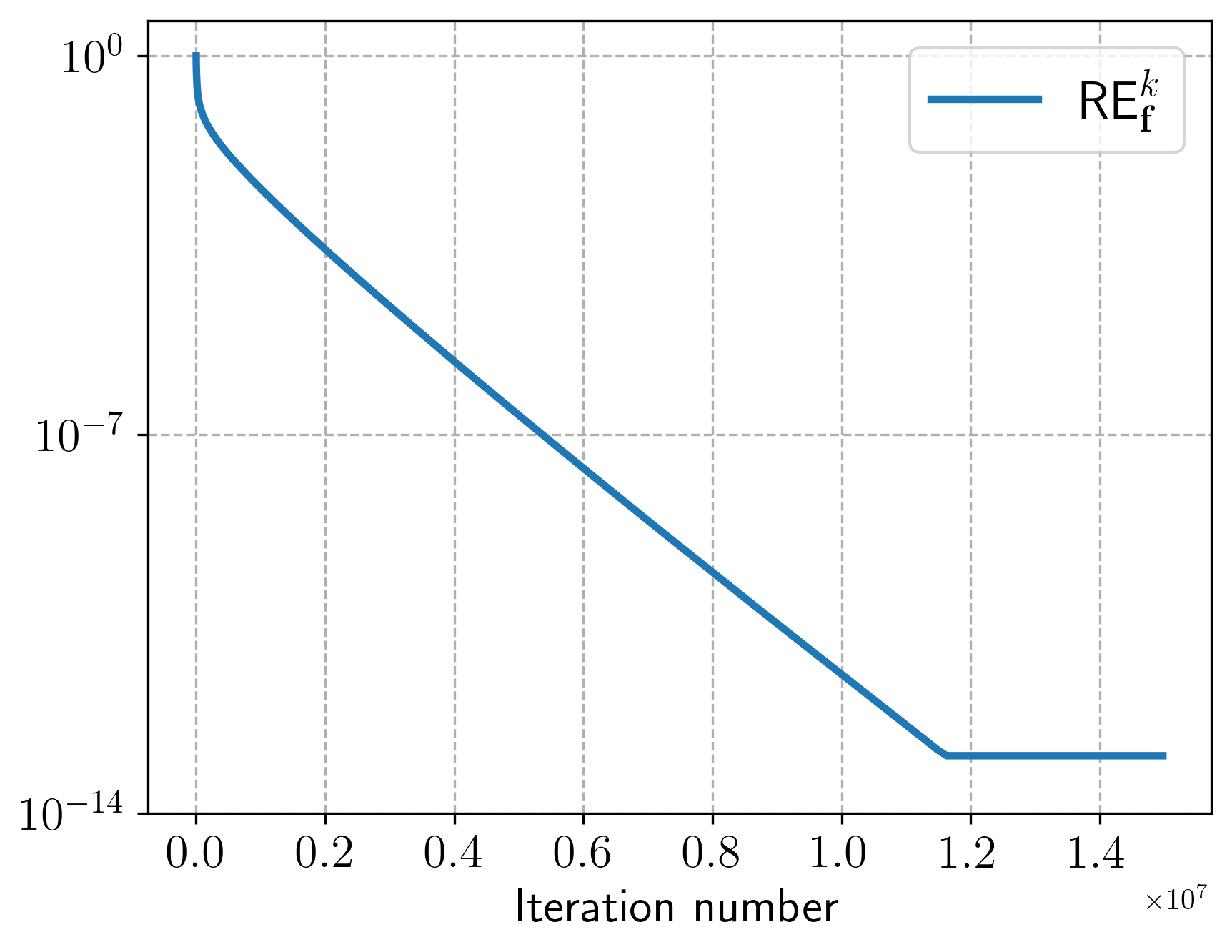}
    \includegraphics[width=0.45\textwidth]{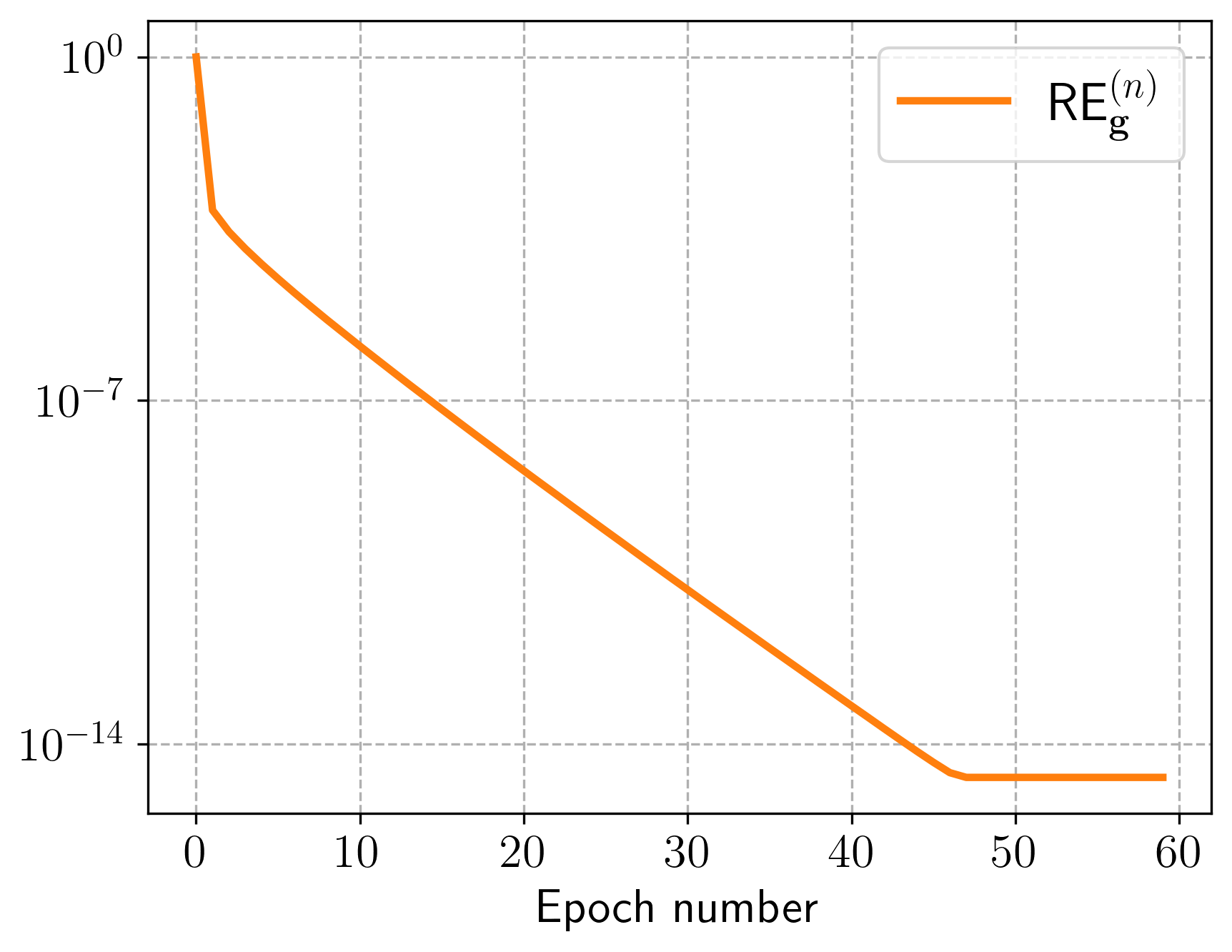}
        \vspace{3mm}
    \caption{The metrics RE$_{\mathbf{f}}^k$ \cref{eq:REk_f} (left) and RE$_{\mathbf{g}}^{(n)}$ \cref{eq:REn_g} (right) are plotted in a semi-log scale as a function of iteration number and epoch number for the reconstruction of the forbild head phantom obtained by the NKM using the maximum residual strategy, respectively.}
    \label{fig:REk_GIC}
\end{figure}

Furthermore, the basis images and VMIs obtained at 60 keV and 100 keV are displayed in \cref{fig:recon_forbild}, which are quite close to the truths by visual comparison.
\begin{figure}[htbp]
    \centering
    \includegraphics[width=0.95\textwidth]{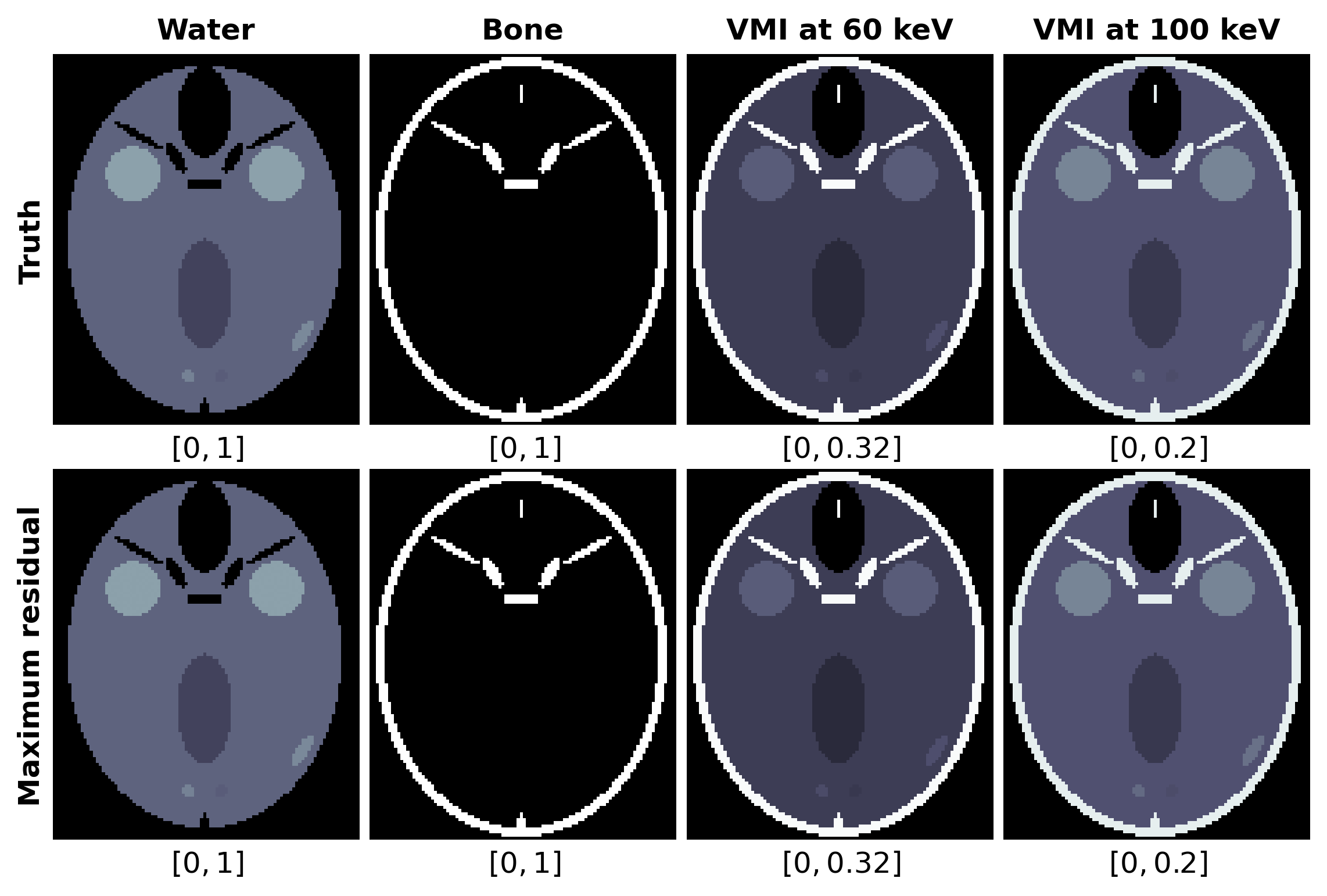}
        \vspace{3mm}
    \caption{The basis images of water (column 1) and bone (column 2) and VMIs at energies 60 keV (column 3) and 100 keV (column 4), respectively. From top to bottom, the truths (row 1), the reconstructed results obtained by the NKM using the maximum residual strategy after 60 epochs (row 2), respectively.}
    \label{fig:recon_forbild}
\end{figure}

\section{Conclusion}
\label{sec:Discussion}

In this work, we investigated the convergence of the NKM for solving the systems of nonlinear equations with component-wise convex mapping, where the function corresponding to each equation is convex. Such kind of nonlinear systems have been emerged widely in various practical applications, including MSCT, discrete X-ray transform with nonlinear partial volume effect, phase retrieval, and many others, which constitute a class of nonlinear imaging inverse problems. 

It is known that the component-wise local TCC has been commonly used in the convergence analysis for the NKM. We showed that in the finite-dimensional space the nonlinear system satisfied such a condition would degenerate to be linear in the local region, and the component-wise convex mapping of interest may not satisfy the component-wise local TCC. To this end, we proposed a novel condition named {\it relative gradient discrepancy condition} (RGDC), and made use of it to prove the convergence and even the convergence rate of the NKM with several general index selection strategies. In particular, these strategies include the often used cyclic strategy and maximum residual strategy. 

Specifically, we further studied the application of the NKM for solving the nonlinear system in MSCT image reconstruction. We proved that the relevant nonlinear mapping fulfills the proposed RGDC rather than the component-wise local TCC, and provided the global convergence results for the NKM based on the previously established theoretical results. Numerical experiments demonstrated the numerical convergence of the NKM for solving the problems of interest and validated our theoretical results. Note that the proposed RGDC is relatively easy to verify.

To the best of our knowledge, this is the first literature on the convergence analysis of the NKM without using local TCC when solving the component-wise convex nonlinear system. 

\begin{appendices}
\section{}
\label{sec:appendix}

\subsection{Proof of \cref{lemma:RGDC_bound}}\label{proof:RGDC_bounded}
\begin{proof}
    Since the points $\bx_1,\bx_2$ are arbitrary in \cref{ass:relative_grad_discrepancy}, one can swap their positions, and then sum the two inequalities to obtain
    \begin{multline*}
        \|\nabla F_j(\bx_1)\|^2 - 2 \langle \nabla F_j(\bx_1), \nabla F_j(\bx_2) \rangle + \|\nabla F_j(\bx_2)\|^2 =
        \| \nabla F_j(\bx_1) - \nabla F_j(\bx_2) \|^2 \\
        \le \frac{\gamma^2}{2} \left(\|\nabla F_j(\bx_1)\|^2 + \|\nabla F_j(\bx_2)\|^2\right).
    \end{multline*}
    Arranging the above terms and using Cauchy--Schwarz inequality yield that
    \begin{align*}
        \left\langle \nabla F_j(\bx_1), \nabla F_j(\bx_2) \right\rangle &\ge \frac{1}{2} \left(1-\frac{\gamma^2}{2}\right) \left( \|\nabla F_j(\bx_1)\|^2 + \|\nabla F_j(\bx_2)\|^2 \right) \\
        &\ge \left(1-\frac{\gamma^2}{2}\right) \|\nabla F_j(\bx_1)\|\cdot \|\nabla F_j(\bx_2)\|. 
    \end{align*}
    Then the result \cref{eq:angle_bounded} immediately follows. Furthermore, using triangle inequality, \cref{eq:length_bounded} can also be obtained.
\end{proof}

\subsection{Proof of \cref{lemma:RGDC_property}}\label{proof:RGDC}
\begin{proof}
    Fixing $\bx_1$, by \cref{ass:relative_grad_discrepancy}, we obtain that for any $\bx\in\Omega$
    \begin{align}
        \|\bF^{\prime}(\bx_1) - \bF^{\prime}(\bx)\|_{F} =& \sqrt{\sum_{j=1}^J \|\nabla F_j(\bx_1) - \nabla F_j(\bx)\|^2 } \nonumber \\
        \le& \sqrt{\gamma^2 \sum_{j=1}^J \|\nabla F_j(\bx_1)\|^2 } =\gamma \|\bF^{\prime}(\bx_1)\|_{F}. \label{eq:F_norm_gamma}
    \end{align}
    Since $\bF^{\prime}(\bx_1)$ has full column rank, we can write its QR decomposition as 
    \begin{equation*}
       \bF^{\prime}(\bx_1) = Q \begin{bmatrix}
        R \\
        0
       \end{bmatrix},
    \end{equation*}
    where $Q\in\Real^{J\times J}$ is an orthogonal matrix, and $R\in\Real^{N\times N}$ is an invertible upper-triangular matrix. Note that they are both related to $\bx_1$. For any $\bx\in\Omega$, define the following matrix 
    \begin{align*}
        A(\bx_1,\bx) = 
        \begin{bmatrix}
            \left(\bF^{\prime}(\bx) - \bF^{\prime}(\bx_1) \right)R^{-1}  & 0
        \end{bmatrix}
        Q^{\tra}\in\Real^{J\times J}.
    \end{align*} 
    Then it is easy to show that
    \begin{align*}
        A(\bx_1,\bx) \bF^{\prime}(\bx_1) = \bF^{\prime}(\bx) - \bF^{\prime}(\bx_1). 
    \end{align*}
    The properties of the matrix norm and \cref{eq:F_norm_gamma} yield that
    \begin{multline*}
        \|A(\bx_1,\bx)\| = \|\left(\bF^{\prime}(\bx) - \bF^{\prime}(\bx_1) \right)R^{-1}\| \le\| \bF^{\prime}(\bx) - \bF^{\prime}(\bx_1)\| \cdot \|R^{-1}\| \\
        = \| \bF^{\prime}(\bx) - \bF^{\prime}(\bx_1)\| \cdot \|\bF^{\prime}(\bx)^{\dagger}\|
        \le \| \bF^{\prime}(\bx) - \bF^{\prime}(\bx_1)\|_F \cdot \|\bF^{\prime}(\bx)^{\dagger}\| \\
        \le \gamma \|\bF^{\prime}(\bx_1)\|_F \cdot \|\bF^{\prime}(\bx_1)^{\dagger}\| = \gamma\kappa_F(\bF^{\prime}(\bx_1)).
    \end{multline*}
    Using the above results and mean value theorem, it follows that
    \begin{align*}
    \|\bF(\bx_1) - &\bF(\bx_2)-\bF^{\prime} (\bx_1 ) (\bx_1-\bx_2 )\| \\
    &=\left\| \int_0^1 \bF^{\prime}(\bx_t )(\bx_1-\bx_2 ) -\bF^{\prime}(\bx_1 ) (\bx_1-\bx_2 )\ \mathrm{d}t \right\| \\
    &=\left\| \int_0^1   A(\bx_1,\bx_t) \bF^{\prime}(\bx_1 ) (\bx_1-\bx_2 )\ \mathrm{d}t \right\| \\
    &\le  \int_0^1  \|A(\bx_1,\bx_t)\| \cdot \|\bF^{\prime}(\bx_1 ) (\bx_1-\bx_2)\|\  \mathrm{d}t   \\
    &\le  \gamma\kappa_F(\bF^{\prime}(\bx_1)) \|\bF^{\prime}(\bx_1 ) (\bx_1-\bx_2)\|,
    \end{align*}
    where $\bx_t := t\bx_1 + (1-t)\bx_2$. Note that since $\Omega$ is a convex set, $\bx_t\in\Omega$. Thus mean value theorem is available and the \ref{ass:relative_grad_discrepancy} holds at points $\bx_1,\bx_t$.
\end{proof}

\subsection{Proof of \cref{lemma:unique_solution}}\label{proof:RGDC_unique}
\begin{proof}
    By \cref{lemma:RGDC_property}, we know that \cref{eq:derived_nonlin_condi} holds for $\bx^*$ and any $\bx\in\Omega$. Assume that there exists another solution $\bar{\bx}^*\in\Omega$. Taking $\bar{\bx}^*$ into \cref{eq:derived_nonlin_condi} implies that
    \begin{equation*}
        \|\bF(\bx^*) - \bF(\bar{\bx}^*)-\bF^{\prime}(\bx^*) (\bx^*-\bar{\bx}^*)\| \le
        \gamma\kappa_F(\bF^{\prime}(\bx^*)) \|\bF^{\prime}(\bx^* ) (\bx^*-\bar{\bx}^*)\|.
    \end{equation*} 
    Since $\gamma\kappa_F(\bF^{\prime}(\bx^*)) < 1$, the right-hand side of the above inequality has to be $0$. Since $\bF^{\prime}(\bx^*)$ has full column rank, it follows that $\bar{\bx}^* = \bx^*$.
\end{proof}

\subsection{Proof of \cref{lem:NKM_convex}.}
\begin{proof}\label{proof:NKM_convex}
    Using the convexity and the iteration procedure in \cref{eq:non_kacz}, we obtain the following result immediately
    \begin{equation*}
            F_{j_k}(\bx^{k+1})  \geq F_{j_k}(\bx^{k}) + \nabla F_{j_k}(\bx^{k})^{\tra}(\bx^{k+1}-\bx^{k}) = 0. 
    \end{equation*}
\end{proof}

\subsection{Proof of \cref{lem:dist_decrease}.}
\begin{proof}\label{proof:dist_decrease}
    When $F_{j_k}(\bx^{k})> 0$, straightforward calculation leads to
    \begin{multline*}
        \| \bx^{k+1} -\bx^* \|^2- \| \bx^{k}-\bx^* \|^2
        = \| \bx^{k+1}-\bx^{k} \|^2+2 \langle\bx^{k+1}-\bx^{k}, \bx^{k}-\bx^* \rangle \\
        = \left\|\frac{F_{j_k}(\bx^{k})}{ \|\nabla F_{j_k} (\bx^{k} ) \|^2} \nabla F_{j_k} (\bx^{k} ) \right\|^2 + 2 \left\langle\frac{-F_{j_k}(\bx^{k})}{ \|\nabla F_{j_k} (\bx^{k} ) \|^2} \nabla F_{j_k} (\bx^{k} ), \bx^{k}-\bx^* \right\rangle \\
        = \frac{ (F_{j_k}(\bx^{k}) )^2}{ \|\nabla F_{j_k} (\bx^{k} ) \|^2}+ \frac{2 F_{j_k} (\bx^{k} )}{ \|\nabla F_{j_k} (\bx^{k} ) \|^2} \left[F_{j_k} (\bx^{k} )-F_{j_k} (\bx^* )-\nabla F_{j_k} (\bx^{k} )^{\tra} (\bx^{k}-\bx^* ) \right] \\
        \quad - \frac{2 (F_{j_k} (\bx^{k} ))^2}{ \|\nabla F_{j_k} (\bx^{k} ) \|^2}  \\
        = \frac{F_{j_k} (\bx^{k} ) }{ \|\nabla F_{j_k} (\bx^{k} ) \|^2} \Big\{ 2  \left[F_{j_k} (\bx^{k}) -F_{j_k} (\bx^*)-\nabla F_{j_k} (\bx^{k} )^{\tra} (\bx^{k}-\bx^* )\right] -(F_{j_k}(\bx^{k}) ) \Big\}.
    \end{multline*}
    The convexity of $F_{j_k}(\bx)$ yields that
    \begin{align*}
        F_{j_k} (\bx^{k}) -F_{j_k} (\bx^*)-\nabla F_{j_k} (\bx^{k} )^{\tra} (\bx^{k}-\bx^* )\le 0.
    \end{align*}  
    Together with $F_{j_k}(\bx^{k})> 0$, the desired result follows.
\end{proof}
\subsection{Proof of \cref{lem:RGDC_result}.}
\begin{proof}\label{proof:RGDC_result}
According to the conclusion in \cref{lem:dist_decrease}, we obtain $\{\bx^k\}$ is always in $\mathcal{B}_{\rho}(\bx^*)$. By mean value theorem, it follows that 
    \begin{align*}
        |F_{j_k}(\bx^{k+1})| &= |F_{j_k}(\bx^{k}) + \nabla F_{j_k}(\bx_{\xi})^{\tra}(\bx^{k+1}-\bx^{k})|\\
        &=\left|F_{j_k}(\bx^{k}) -\frac{F_{j_k}(\bx^{k})}{\|\nabla F_{j_k}(\bx^{k})\|^2} \nabla F_{j_k}(\bx_{\xi})^{\tra}\nabla F_{j_k}(\bx^{k})\right| \\
        &= |F_{j_k}(\bx^{k})|\cdot \left|1-\frac{ \nabla F_{j_k}(\bx_{\xi})^{\tra}\nabla F_{j_k}(\bx^{k})}{\|\nabla F_{j_k}(\bx^{k})\|^2}\right|, 
    \end{align*}
    where $\bx_{\xi} = \xi\bx^{k} + (1-\xi)\bx^{k+1}\in\mathcal{B}_{\rho}(\bx^*)$, $0\le \xi\le 1$. Notice that
    \begin{align*}
        \frac{ \nabla F_{j_k}(\bx_{\xi})^{\tra}\nabla F_{j_k}(\bx^{k})}{\|\nabla F_{j_k}(\bx^{k})\|^2} = \frac{\nabla F_{j_k}(\bx_{\xi})^{\tra}\nabla F_{j_k}(\bx^{k})}{\|\nabla F_{j_k}(\bx_{\xi})\|\cdot \|\nabla F_{j_k}(\bx^{k})\|} \cdot \frac{\|\nabla F_{j_k}(\bx_{\xi})\|}{\|\nabla F_{j_k}(\bx^{k})\|}.
    \end{align*}
    Hence, using \cref{eq:angle_bounded} and \cref{eq:length_bounded} proved in \cref{lemma:RGDC_bound}, we have
    \begin{align*}
        1\cdot(1+\gamma)\ge \frac{\nabla F_{j_k}(\bx_{\xi})^{\tra}\nabla F_{j_k}(\bx^{k})}{\|\nabla F_{j_k}(\bx_{\xi})\|\cdot \|\nabla F_{j_k}(\bx^{k})\| } \cdot \frac{\|\nabla F_{j_k}(\bx_{\xi})\|}{\|\nabla F_{j_k}(\bx^{k})\|} \ge \left(1-\frac{\gamma^2}{2}\right)\cdot (1-\gamma).
    \end{align*}
    Consequently, the result follows by
    \begin{align*}
        |F_{j_k}(\bx^{k+1})| &\le |F_{j_k}(\bx^{k})|\cdot \max\left\{|1-(1+\gamma)|,|1-\left(1-\frac{\gamma^2}{2}\right)\cdot (1-\gamma)| \right\} \\
        &= \left(\gamma+\frac{\gamma^2}{2}-\frac{\gamma^3}{2}\right) |F_{j_k}(\bx^{k})|.
    \end{align*}
    This completes the proof. 
    \end{proof}

\subsection{Verification of component-wise convexity for the mapping in geometrically-consistent MSCT}
\label{subsec:verify_convex}

\begin{lemma}\label{lem:map_MSCT_convex}
    Assume that $s_m^{[p]}\ge0$ and $b_{dm}>0$ for all $m,p,d$. The mapping $\bH(\bz)$ defined in \cref{eq:nonsys_GC_nkm} is component-wise convex with respect to $\bz$. Furthermore, if $\boldsymbol{1}\notin Range(B^{\tra})$ and $B$ has full row rank, where $\boldsymbol{1} := [1, 1, \ldots, 1]^{\tra}$, then $\bH(\bz)$ is component-wise strictly convex. 
\end{lemma}
\begin{proof}
Firstly, $\bH$ is continuously differentiable. The gradient of $H_p(\bz)$ is given by
\begin{align*}
    \nabla H_p(\bz) = -B \bdw^{[p]}(\bz), 
\end{align*}
where 
\begin{equation}\label{eq:def_wx}
    \bdw^{[p]}(\bz) = \frac{\bds^{[p]}\odot \bzeta(\bz)}{  {\bds^{[p]}}^{\tra} \bzeta(\bz)},\quad  \bds^{[p]} = [s_1^{[p]},\ldots,s_M^{[p]}]^{\tra}\in\Real^M,
\end{equation}
and 
\begin{align*}
    \bzeta(\bz) = [\zeta_{1}(\bz), \ldots, \zeta_{M}(\bz)]^{\tra},\quad
    \zeta_{m}(\bz) =\exp\left(-\sum_{d=1}^D b_{dm} z_{d}\right).
\end{align*}
Here $\odot$ represents the operation of the element-wise product. We abbreviate $\bdw^{[p]}(\bz)$ as $\bdw$ for simplicity below. Then the Hessian matrix of $H_p(\bz)$ is given by
\begin{align*}
    \nabla^2 H_p(\bz) = B \diag(\bdw) B^{\tra} - (B \bdw) (B \bdw)^{\tra} = B \big(\diag(\bdw)-\bdw\bdw^{\tra}\big) B^{\tra}.
\end{align*}
Since $\bdw\ge 0$ and $\boldsymbol{1}^{\tra}\bdw = 1$, $\forall \bdv\in\Real^M$,
\begin{align*}
    \bdv^{\tra} \big(\diag(\bdw)-\bdw\bdw^{\tra}\big) \bdv &= (\boldsymbol{1}^{\tra}\bdw) \langle \bdv\odot\sqrt{\bdw}, \bdv\odot\sqrt{\bdw} \rangle - \langle \bdv, \bdw \rangle^2 \\
    &= \langle \sqrt{\bdw},\sqrt{\bdw}\rangle \langle \bdv\odot\sqrt{\bdw}, \bdv\odot\sqrt{\bdw} \rangle - \langle \sqrt{\bdw}, \bdv\odot\sqrt{\bdw} \rangle^2\\
    &\ge 0.
\end{align*}
The last inequality follows from the Cauchy--Schwarz inequality, the equal sign holds if and only if $\bdv\odot\sqrt{\bdw}$ and $\sqrt{\bdw}$ are linear dependent, which implies that $\bdv = \nu\boldsymbol{1}$, $\nu\in\Real$. Then for any $\bdu\in\Real^D$, 
\begin{align*}
    \bdu^{\tra} \nabla^2 H_p(\bz) \bdu &= (B^{\tra}\bdu)^{\tra} \big(\diag(\bdw)-\bdw\bdw^{\tra}\big) (B^{\tra}\bdu)\ge0,
\end{align*}
it follows that $\nabla^2 H_p(\bz)\succeq 0$, then $H_p(\bz)$ is convex. Moreover, if $\boldsymbol{1}\notin Range(B^{\tra})$ and $B$ has full column rank, there is no $\bdu$ such that $B^{\tra}\bdu = \nu\boldsymbol{1}$, and $H_p(\bz)$ is strictly convex. 
\end{proof}

\subsection{Verification of the RGDC for the mapping in geometrically-consistent MSCT}\label{subsec:verify_RGDC}
\begin{lemma}\label{lem:general_verify_grad}
    Let $\bH(\bz)$ be defined in \cref{eq:nonsys_GC_nkm}. Assume that  
    \begin{equation}\label{eq:gamma_general}
        \gamma_B := \max\limits_{\bdw_1,\bdw_2\in\Delta_M} \frac{\|B\bdw_1 - B\bdw_2\|}{\|B\bdw_1 \|}<1, 
    \end{equation}
    where $\Delta_M$ is the unit simplex 
\begin{align*}
\left\{\bdw \mid w_m\ge 0,\sum_{m=1}^M w_m = 1  \right\}.
\end{align*} 
    Then $\bH(\bz)$ satisfies the \ref{ass:relative_grad_discrepancy} for the above $\gamma_B$ in $\Real^D$.
\end{lemma}
\begin{proof}
    By the formula in \cref{eq:def_wx}, it follows that for any $\bz\in\Real^D$, $\bdw^{[p]}(\bz)\in\Delta_M$. Since $\|\nabla H_p(\bz)\|$ never vanishes, for any $\bz_1,\bz_2\in\Real^D$, we have
\begin{align*}
    \frac{\| \nabla H_p(\bz_1) - \nabla H_p(\bz_2) \|}{\|\nabla H_p(\bz_1)\|} = \frac{\|B\bdw^{[p]}(\bz_1) - B\bdw^{[p]}(\bz_2)\| }{\|B\bdw^{[p]}(\bz_1)\|}\le \gamma_B.
\end{align*}
\end{proof}

Furthermore, we give a specific example of $\bH$ such that the \ref{ass:relative_grad_discrepancy} holds.
\begin{lemma}\label{lem:specific_verify_grad}
Let $\bH(\bz)$ be defined in \cref{eq:nonsys_GC_nkm}. 
Assume that for every $1\le d\le D$, there exists $m_1,m_2$ such that $b_{d{m_1}}<b_{dm}$ when $m\neq m_1$ and 
$b_{d{m_2}}>b_{dm}$ when $m\neq m_2$, respectively. If
\begin{align*}
    \widetilde{\gamma}_B := \sqrt{ \frac{\sum_{d=1}^D \left( b_{d{m_1}} - b_{d{m_2}} \right)^2 }{\sum_{d=1}^D \left( b_{d{m_1}} \right)^2 } } <1,    
\end{align*} 
then $\bH(\bz)$ satisfies the \ref{ass:relative_grad_discrepancy} for the above $\widetilde{\gamma}_B$ in $\Real^D$.  
\end{lemma}

\begin{proof}
By the assumption of $B$, we have 
\begin{align*}
    \|B\bdw^{[p]}(\bz)\| \ge \|B\bde_{m_1}\| = \sqrt{\sum_{d=1}^D (b_{d{m_1}})^2 }, \quad\forall\bz \in \Real^D.
\end{align*}
Since $B\bdw>0$ for $\bdw\in\Delta_M$, for any $\bz_1,\bz_2$,  
\begin{align*}
    \frac{\| \nabla H_p(\bz_1) - \nabla H_p(\bz_2) \|}{\|\nabla H_p(\bz_1)\|} \le
    \frac{\|B\bde_{m_1} - B\bde_{m_2}\|}{\|B\bde_{m_1}\|} = \widetilde{\gamma}_B.
\end{align*}
\end{proof}

\subsection{Mappings in MSCT do not satisfy the component-wise local TCC}
\begin{proposition}\label{thm:map_MSCT_not_tcc}
Assume that $s_m^{[p]}\ge0$ and $b_{dm}>0$ for all $m,p,d$. If $\boldsymbol{1}\notin Range(B^{\tra})$ and $B$ has full row rank, then $\bH$ and $\bdK$ defined in \cref{eq:nonsys_GC_nkm,eq:nonsys_GIC_nkm} do not satisfy the component-wise local TCC in  \cref{eq:local_tcc_component}.
\end{proposition}
\begin{proof}
Using \cref{lem:map_MSCT_convex}, we know that if $\boldsymbol{1}\notin Range(B^{\tra})$ and $B$ has full row rank, then $H_p$ is strictly convex and all eigenvalues of its Hessian matrix are positive. Furthermore, since $\bda_j^{[p]} (\bda_j^{[p]})^{\tra}$ has only one positive eigenvalue. Then according to the proof of \cref{thm:kac_GIC_converge}, $\nabla^2 K_j^{[p]}(\mathbf{f})$ would have the same number of positive eigenvalues as $\nabla^2 H_p(\bz_j^{[p]})$. That is, the Hessian at least have two positive eigenvalues. Thus, \cref{eq:curvature} holds. By \cref{thm:nonzero_curvature_tcc}, we obtain that $\bH$ and $\bdK$ do not satisfy the component-wise local TCC.
\end{proof}

\end{appendices}          

\bibliography{MCTreferences}

\end{document}